\documentclass[10pt,a4paper]{article}

\usepackage{latexsym,amsfonts,amsmath,amsthm,amssymb, graphicx, mathrsfs, fancyhdr, makeidx, amscd,  color,  fontenc, stix}
\usepackage[all]{xy}

\newtheorem{theorem}{Theorem}[section]

\newtheorem{definition}[theorem]{Definition}
\newtheorem{example}[theorem]{Example}

\newtheorem{lemma}[theorem]{Lemma}

\newtheorem{proposition}[theorem]{Proposition}
\newtheorem{remark}[theorem]{Remark}

\newcommand{\RCD}{\mathsf{RCD}}
\newcommand{\R}{\mathbb{R}}

\newcommand{\metric}{\langle \, , \, \rangle}

\newcommand{\disp}{\displaystyle}

\newcommand{\eps}{\varepsilon}

\newcommand{\Sph}{\mathbb{S}}
\newcommand{\di}{\mathrm{d}}

\newcommand{\Ric}{\mathrm{Ric}}
\newcommand{\Sec}{\mathrm{Sec}}

\newcommand{\lip}{\mathrm{Lip}}
\newcommand{\loc}{\mathrm{loc}}

\newcommand{\LL}{\mathscr{L}}

\DeclareMathOperator{\diver}{div\,}
\DeclareMathOperator{\dist}{\di}

\begin{document}

\author{Dami\~ao J. Ara\'ujo \and Marco Magliaro \and Luciano Mari \and Leandro F. Pessoa}
\title{\textbf{On splitting complete manifolds via infinity harmonic functions}}
\date{}
\maketitle
\scriptsize \begin{center} Departamento de Matem\'atica, Universidade Federal da Para\'iba\\
58059-900, Jo\~ao Pessoa - Para\'iba (Brazil)\\
E-mail: araujo@mat.ufpb.br
\end{center}

\scriptsize \begin{center} Departamento de Matem\'{a}tica, Universidade Federal do Cear\'a\\
60455-760, Fortaleza, (Brazil)\\
E-mail: marco.magliaro@mat.ufc.br
\end{center}

\scriptsize \begin{center} Dipartimento di Matematica, Universit\`a degli studi di Torino,\\
Via Carlo Alberto 10, 10123 Torino (Italy)\\
E-mail: luciano.mari@unito.it
\end{center}

\scriptsize \begin{center} Departamento de Matem\'{a}tica, Universidade Federal do Piau\'{i},\\
64049-550, Teresina (Brazil)\\
E-mail: leandropessoa@ufpi.edu.br
\end{center}

\normalsize

\maketitle

\begin{abstract}
In this paper, we prove some splitting results for manifolds supporting a non-constant infinity harmonic function which has at most linear growth on one side. Manifolds with non-negative Ricci or sectional curvature are considered. In dimension $2$, we extend Savin's theorem on Lipschitz infinity harmonic functions in the plane to every surface with non-negative sectional curvature.
\end{abstract}

\section{Introduction}
The present paper regards the interplay between the geometry of a Riemannian manifold and the qualitative properties of $\infty$-harmonic functions, i.e., solutions to 
	\[
	\Delta_\infty u \doteq \nabla^2 u(\nabla u, \nabla u) = 0 \quad \text{on } \, M
	\]
in the viscosity sense. The $\infty$-Laplace operator and its normalized counterpart
	\[
	\Delta^N_\infty u \doteq \nabla^2 u\left(\frac{\nabla u}{|\nabla u|}, \frac{\nabla u}{|\nabla u|}\right)
	\]
gained increasing importance in the field of fully-nonlinear PDEs over the past 60 years, see \cite{aronsson_crandall_juutinen,crandall_visit} for a thorough account of the theory, with historical insights and a detailed set of references. The investigation herein is a natural continuation of \cite{amp,maripessoa,maripessoa_2}, where the geodesic completeness of a boundaryless Riemannian (or Finsler) manifold was characterized in terms of suitable Liouville properties of viscosity solutions to
\[
\Delta^N_\infty u \ge g(u).
\]
It is known that $u$ solves $\Delta_\infty u \ge 0$ ($=0$, $\le 0$) if and only if it solves $\Delta_\infty^N u \ge 0$ ($=0$, $\le 0$). Therefore, for the purpose of the present paper we will only consider $\Delta_\infty$. Among the various equivalent conditions, by \cite[Theorem 1.1]{amp} (cf. also \cite[Theorem 8.1]{maripessoa}) a connected Riemannian manifold $M$ without boundary is shown to be complete \emph{if and only if} all solutions to $\Delta_\infty u \le 0$ whose negative part $u_-$ satisfies 
\begin{equation}\label{eq_slowerlin}
u_-(x) = o\big(r(x)\big) \qquad \text{as } \, x \to \infty
\end{equation}
are constant\footnote{The implication is $(1) \Leftrightarrow (2)$ in \cite[Theorem 1.1]{amp}, once we observe that $v = -u$ solves $\Delta_\infty v \ge 0$ (hence $\Delta^N_\infty v \ge 0$) with $v_+ = o(r)$.}. Here, $r(x)$ denotes the distance to a fixed origin. The result extends the known Liouville theorem for positive $\infty$-superharmonic functions on $\R^m$ proved by Lindqvist and Manfredi in \cite{lindqvistmanfredi, lindqvistmanfredi_2}, see also \cite[p.113]{crandall_visit}, and we stress that the \emph{if} part is its main novelty. The lack of curvature or volume growth requirements on $M$ in order for the aforementioned Liouville property to hold makes the theory of slowly growing $\infty$-harmonic functions considerably different from that developed for other operators $\mathcal{F}$ like the Laplacian \cite{cheng,yau}, the $p$-Laplacian \cite{wangzhang} and in recent years the minimal hypersurface operator \cite{cmmr,ding,rosenbergschulzespruck}. In these latter cases, $\Ric \ge 0$ is the weakest known condition to guarantee that positive solutions to $\mathcal{F}[u]=0$ are constant. For solutions satisfying the more general \eqref{eq_slowerlin}, in the minimal hypersurface case further technical conditions on $M$ are needed as of yet, see \cite{cgmr,ding_new}.

Hereafter, $M$ will always denote a complete, connected Riemannian manifold without boundary. A natural problem is then to see what happens to $\infty$-harmonic functions that grow at most linearly on one side, namely, that satisfy
	\begin{equation}\label{eq_noti}
	\limsup_{r(q)\to \infty} \frac{u(q)}{r(q)} < \infty.
	\end{equation}
Especially, we shall look for geometric conditions to force a rigidity of $M$ or $u$, in the sense that $M$ splits as a (possibly warped) product and $u$ only depends on split-off variables. The next example shows that a constraint on the geometry of $M$ is necessary in this case.

\begin{example}
\emph{On a Cartan-Hadamard manifold, that is, a simply connected manifold with nonpositive sectional curvature $\Sec$, given a ray $\gamma : [0,\infty) \to M$ one can consider the Busemann function
	\[
	b_\gamma : M \to \R, \qquad b_\gamma(x) = \lim_{t \to \infty} \Big( \di(x,\gamma(t)) - t\Big). 
	\]
It is known by \cite{heintze-hof} that $b_\gamma \in C^2(M)$ and $|\nabla b_\gamma| = 1$ on $M$, so differentiating we get that $b_\gamma$ is a globally Lipschitz solution to $\Delta_\infty b_\gamma = 0$ on $M$. However, in general $M$ does not split off any line.
}
\end{example}
\begin{remark}\label{rem_limit_1}
\emph{It is known, see Lemma \ref{lip_const_S_R} below, that for solutions to $\Delta_\infty u \ge 0$ the following identity holds (possibly with infinite values):
	\[
	\limsup_{r(q)\to \infty} \frac{u(q)}{r(q)} = \lip(u,M),
	\]
with $\lip(u,M)$ the Lipschitz constant of $u$ on $M$. Therefore, non-constant globally Lipschitz solutions to $\Delta_\infty u = 0$ are precisely those for which the limsup in \eqref{eq_noti} is a positive real number. By scaling $u$, in our main results we shall assume this number to be one.  
}
\end{remark}
Based on the theory of harmonic functions with linear growth developed in \cite{ccm,kasue, litam} and the corresponding results for minimal graphs which appeared in recent years \cite{cgmr,ding_new,dingjostxin}, the assumptions
	\[
	\Ric \ge 0 \qquad \text{or} \qquad \Sec \ge 0
	\] 
seem to be appropriate. For Euclidean space, Aronsson in \cite[Section 7]{aronsson2} proved that any solution of class $C^2$ on $\R^2$ is affine, see also \cite{evans} for the case of dimension $m \ge 3$. Examples therein show that this fails for viscosity solutions which are not $C^2$, unless one assumes a priori growth of $u$. On the other hand, Savin's remarkable theorem \cite{savin} states that  
	\[
	\Delta_\infty u = 0 \ \ \ \text{on } \, \R^2, \ \ \ \text{$u$ \ Lipschitz} \quad \Longrightarrow \quad \text{$u$ \ is\ affine}.
	\] 
As of today, its extension to $\R^m$ for $m \ge 3$ has not been established. In higher dimensions, we are only aware of the next half-space theorem showed by Crandall, Evans and Gariepy \cite{CEG}:
	\[
	\Delta_\infty u \le 0, \ \ \  u(x) \ge a + \langle p, x \rangle \ \text{ on } \, \R^m \quad \Longrightarrow \quad u = u(0) + \langle p, x \rangle ,
	\]
and the recent work of Hong and Zhao \cite{hongzhao}, who proved that $u$ is affine by assuming \eqref{eq_noti} and 
	\[
	\lim_{r(p) \to \infty} |Du(p)| = \lip(u,\R^m).
	\]
The methods herein are much inspired by those in \cite{CEG,hongzhao}. As a matter of fact, we show that elaborating on their arguments in a manifold setting, and employing some basic facts of metric geometry, we are able to obtain results with nontrivial geometric content. Let $M$ be complete, connected and without boundary, and assume that $u$ is a nonconstant $\infty$-harmonic function satisfying \eqref{eq_noti}, so by Remark \ref{rem_limit_1} we can assume 
\[
\limsup_{r(q) \to \infty} \frac{u(q)}{r(q)} =1.
\]
We prove:
\begin{itemize}
\item[$(i)$] Theorem \ref{teo_ric}. If $\Ric \ge 0$, then any blowdown $M_\infty$ of $M$ splits as $\R \times N_\infty$. Moreover, the blowdown of $u$ only depends on the arclength $t$ of the $\R$ factor, and it is affine in $t$. 

\item[$(ii)$] Proposition \ref{teo_counter}. In the assumptions of $(i)$, $M$ itself may not split off lines: there exists a manifold $M$ with $\Ric > 0$ carrying a linearly growing $\infty$-harmonic function. However, by the tangency principle in Proposition \ref{prop_SMP}, if the graph of $u$ touches that of a (possibly translated and dilated) Busemann function from above or below, then $M$ splits and $u$ is an affine function of the split direction only.

\item[$(iii)$] If $\Sec \ge 0$, general theory gives a way to split $M$ itself as $\R \times N$. We prove in Theorem \ref{teo_unique_blowdown}. that the blowdown of $u$ is unique and that, writing $(x,y) \in \R \times N$ and orienting $\R$ appropriately, it holds
\[
\lim_{x \to +\infty} \frac{u(x,y)-u(-x,y)}{2x} = 1 \qquad \text{for each fixed } \, y \in N.
\]  
\end{itemize}

In the assumptions of $(iii)$, whether the function $u$ only depends on $x$ is an open problem even in $\R^m$, whose solution would allow to extend Savin's result to higher dimensions. As pointed out in \cite{crandall_visit,CEG}, a positive answer is likely to give new insights on the $C^{1,\alpha}$ regularity property of $\infty$-harmonic functions. In dimension $m \ge 3$, we have the following sufficient condition:
\begin{itemize}
\item[$(iv)$] Assume $\Sec \ge 0$, and that there exist a ray $\gamma$ and a constant $C$ for which either
\[
u\big(\gamma(t)\big) \ge t-C \qquad \text{or} \qquad u\big(\gamma(t)\big) \le - t + C 
\]
holds for all $t \in \R^+$. Then, referring to the splitting in $(iii)$, we have $u(x,y) = x + C_2$ for some constant $C_2$.
\end{itemize}

On the other hand, $(iii)$ strengthens in dimension $2$ and gives rise to a full extension of Savin's theorem to any complete surface with non-negative sectional curvature. We get

\begin{theorem}\label{teo_sec_surface_intro}
Let $M$ be a complete connected surface with $\Sec \geq 0$, and let $u \in C(M)$ be a non-constant $\infty$-harmonic function such that
\begin{eqnarray}
\limsup_{r(q) \to \infty} \frac{u(q)}{r(q)} < \infty,
\end{eqnarray}
where $r$ is the distance from a fixed origin. Then, $M = \R^2$ or $M = \R \times \mathbb{S}^1$. Furthermore, $u$ only depends on the arclength $x$ of a split $\R$ factor, and it is affine in $x$. 
\end{theorem}

Most of the arguments in the present paper extend, almost directly, to $\RCD$ spaces, for which we refer to the survey \cite{ambrosio} and the references therein. An exception might be the approximation procedure we carried over to prove Theorem \ref{teo_unique_blowdown}, see Section \ref{sec_pLapla}. As a side remark, to approximate we have chosen to use the $p$-Laplacian instead of the inhomogeneous operator proposed in \cite{evsm}. In another direction, Finsler manifolds proved to be a quite natural setting for the techniques developed to investigate the $\infty$-Laplacian, see \cite{amp}. However, in such a generality the topological/geometric conclusions that can be achieved from splitting theorems are weaker, apart from the subclass of Berwald metrics, \cite{ohta}. For these reasons, we decided to stick to the smooth, Riemannian setting to avoid technicalities.

\vspace{0.2cm}

\noindent \textbf{Acknowledgements.} D.J.A. is partially supported by CNPq-Brazil, grant 311138/2019-5, and Para\'iba State Research Foundation (FAPESQ), grant 2019/0014. D.J.A. and L.P. thank the Abdus Salam International Centre for Theoretical Physics (ICTP). M.M. is partially supported by CNPq-Brazil, Grant 401233/2022-7. L.P. is partially supported by Alexander von Humboldt Foundation, Capes-Brazil (Finance Code 001), and by CNPq-Brazil, Grant 306738/2019-8 and is grateful to Professor Alexander Grigor'yan and the Faculty of Mathematics at the Universit\"at Bielefeld for their warm hospitality. L.M. is supported by the PRIN project 20225J97H5 ``Differential-geometric aspects of manifolds via Global Analysis''.

\section{Preliminaries}

%In this section we are going to describe basic properties on metric geometry, comparison with cones and some qualitative results involving Lipschitz and sub/super-solutions of the infinity Laplace operator.

\subsection*{Busemann functions and convergence}

We here collect some basic facts on metric geometry, mostly to fix notation. We refer to \cite{petersen} for more details. Hereafter, a segment $\gamma : [0,T] \to M$ will be a unit speed geodesic which is minimizing between its endpoints. A unit speed geodesic $\gamma$ will be called: 
\begin{itemize}
\item[-] a ray if $\gamma$ is defined on $[0,\infty)$ and is a segment between any pair of its points; 
\item[-] a line if $\gamma$ is defined on $\R$ and is a segment between any pair of its points.
\end{itemize}
Therefore, a line is characterized by the identity 
\[
\di(\gamma(t), \gamma(s)) = |t-s| \qquad \forall \, s,t \in \R.
\]
Given a ray $\gamma$, the Busemann function $b_\gamma : M \to \R$ is defined as the limit
	\[
	b_\gamma(x) = \lim_{t \to \infty} \Big( \di(x,\gamma(t)) - t\Big). 
	\]
Such a limit exists since the family of functions $b_{\gamma,t}(x) = \di(x,\gamma(t)) - t$ is monotone decreasing and bounded as $t \uparrow \infty$, see \cite[Sec. 7.3.2]{petersen}. Given a point $p \in M$, an asymptote of $\gamma$ issuing from $x$ is a sequential limit $\tilde \gamma$ of a sequence of segments $\tilde \gamma_{j}$ joining $x$ to $\gamma(t_j)$ for some $t_j \to \infty$. Notice that $\tilde \gamma$ is a ray from $x$. By \cite[Prop. 7.3.8]{petersen}, it holds
\[
b_\gamma(x) \le b_\gamma(p) + b_{\tilde \gamma}(x)
\]
with equality at $p$, namely, $b_\gamma(p) + b_{\tilde \gamma}$ is a support function from above for $b_\gamma$ at $p$.

Next, we denote with $\lambda = \{\lambda_j\}$ a sequence with $\lambda_j \to \infty$. For each $j$, we let $M^\lambda_j$ be the manifold $M$ with metric $g_j = \lambda_j^{-2}g$, distance $\di_j = \lambda_j^{-1}\di$ and induced volume form $\di V_j = \lambda_j^{-m}\di V$. Also, let $B_R^j$ be geodesic balls in $M^\lambda_j$ centered at a fixed origin $o$. A pointed measured Gromov-Hausdorff limit
\begin{eqnarray}
(M,\di_j,\di V_j,o) \longrightarrow (M^\lambda_{\infty},\di_{\infty},\mathfrak{m}_{\infty},o_{\infty}),
\end{eqnarray}
see \cite[Section 6]{ambrosio} for its definition, will be written as $M_j^\lambda \to M_\infty^\lambda$ and named a tangent cone at infinity (or a blowdown) of $M$ at $o$. Henceforth, given $u : M \to \R$, $\lip(u,U)$ will denote the Lipschitz constant of $u$ on a subset $U \subset M$. Assume that $u$ is globally Lipschitz. Defining  
	\[
	u^\lambda_j: M^\lambda_j \to R, \qquad u^\lambda_j(x) = \frac{u(x)-u(o)}{\lambda_j}, 
	\]
we have $\lip(u^\lambda_j,M^\lambda_j) = \lip(u,M)$ for each $j$ and therefore, up to a subsequence, $u^\lambda_j$ converges pointwise in the Gromov-Hausdorff sense (see \cite[Lem. 11.1.9]{petersen}) to a function $v^\lambda: M_\infty^\lambda \to \R$, meaning that, for each $x_j \in M^\lambda_j$, $x_\infty \in M^\lambda_\infty$, 
\[
q_j \to q_\infty \qquad \Longrightarrow \qquad u_j^\lambda(q_j) \to v^\lambda(q_\infty).
\]

\subsection*{Comparison with cones and its consequences}

We recall some well-known properties of $\infty$-subharmonic functions, which can be found in the surveys \cite{aronsson_crandall_juutinen,crandall_visit,wang_notes}. Despite the references being set in Euclidean space, the proof of the lemmata below carry over verbatim to any complete Riemannian manifold. For more general metric spaces, we refer to \cite{champion_depascale}. 

Let $\Omega$ be an open domain of $M$ and let $u \in C(\Omega)$. It is known that $\Delta_\infty u \ge 0$ is equivalent to $u$ enjoying comparison with cones 
\[
C_x(y) = a + b\, \di(x,y), \qquad a,b \in \R 
\]
from above, meaning that if $u \le C_x$ in $\partial \Omega \cup \{x\}$, then $u \le C_x$ in $\Omega$, see \cite[Section 3]{CEG} and \cite{champion_depascale,crandall_visit}. As a consequence, if $\Delta_\infty u \ge 0$ then  
\[
u(y) \le u(x) + \left(\max_{\partial B_R(x)} \frac{u(y)-u(x)}{R}\right)\di(x,y) \quad \forall \, x \in \Omega, \ y \in B_R(x) \Subset \Omega.
\] 
Even more, by \cite[Lem. 2.4]{CEG} the function
\begin{eqnarray}
R \mapsto S^+_{u,R}(x) = \max_{z \in \partial B_R(x)} \frac{u(z)-u(x)}{R}
\end{eqnarray}
is non-decreasing for $R < \di(x, \partial \Omega)$, and therefore the limits 
\[
S_u^+(x) = \lim_{R \to 0} S^+_{u,R}(x) \qquad \text{and} \qquad S_{u,\infty}^+(x) = \lim_{R \to \infty} S^+_{u,R}(x) 
\]
(the latter, if $\Omega = M$) are well defined. As a direct consequence, $u \in \lip_\loc(\Omega)$, see \cite[Lem. 2.5]{CEG}. The following proposition collects some of the properties in \cite[Lemm. 4.2 and 4.3]{crandall_visit}.

\begin{proposition}\label{prop_lip_point}
Let $\Omega \subset M$ be an open subset and $u \in C(\Omega)$ satisfy $\Delta_\infty u \ge 0$. Then, for each $x \in \Omega$
\[
S_u^+(x) = \lim_{r \to 0} \lip(u, B_r(x)) = \lim_{r \to 0} \|\nabla u\|_{L^\infty(B_r(x))}.
\]
Moreover, if $u$ is differentiable at $x$, the three quantities equal $|\nabla u(x)|$. 
\end{proposition}

%We next recall the following endpoint Lemma in \cite[Lem. 3.3]{CEG}, see also \cite[Lem. 2.13]{wang_notes}, which will be crucial for us:
%
%\begin{lemma}
%Let $u \in C(\Omega)$ solve $\Delta_\infty u \ge 0$, and let $B_r(x) \Subset \Omega$, $x_r \in \partial B_r(x)$ satisfy 
%\[
%S_{u,r}^+(x) = \frac{u(x_r)-u(x)}{r}.
%\]
%Then, for each $R$ such that $B_R(x_r) \Subset \Omega$, it holds $S_{u,R}^+(x_r) \ge S_{u,r}^+(x)$. In particular, $S_u^+(x_r) \ge S_u^+(x)$.
%\end{lemma}

We next state a simple yet very useful consequence of comparison with cones, essentially contained in  \cite[Prop. 7.1]{crandall_visit} \cite[Prop. 1.1]{hongzhao}. We include a proof for the sake of completeness.

\begin{lemma}\label{lip_const_S_R}
If $u \in C(M)$ satisfies $\Delta_{\infty} u \geq 0$, and let $r$ be the distance from a fixed origin $o$. Then, 
\[
\lip(u,M) = S^+_{u,\infty}(x) = \limsup_{r(q) \to \infty} \frac{u(q)}{r(q)},
\]
for any $x \in M$, possibly with infinite values. 
\end{lemma}
\begin{proof}
We prove the first equality. From the monotonicity of $S_{u,R}^+(x)$ we deduce that $S^+_{u,R}(x) = \max_{z \in \bar B_R(x)} \frac{u(z)-u(x)}{R}$. Therefore, for each $w \in M$ we get
\begin{eqnarray*}
S^+_{u,R}(x) &\leq  & \max_{\bar B_{R+\di(w,x)}(w)} \frac{u(z)-u(x)}{R} \\[0.2cm]
&=& \max_{\partial B_{R+\di(w,x)}(w)} \left(\frac{u(z)-u(w)}{R+\di(w,x)}\right)\frac{R+\di(w,x)}{R} + \frac{u(w)-u(x)}{R} \\[0.2cm]
& \le & \frac{R+\di(w,x)}{R} S^+_{u,R}(w) + \frac{u(w)-u(x)}{R}.
\end{eqnarray*}
Letting $R \to \infty$ we may conclude $S^+_{u,\infty}(x) \leq S^+_{u,\infty}(w)$. Since $x$ and $w$ are arbitrary, equality holds and the limit $\ell = S^+_{u,\infty}(x)$ (possibly infinite) does not depend on the point $x$. We now show that $\ell = \lip(u,M)$. It is clear that $S^+_{u,R}(x) \leq \lip(u,M)$, thus $\ell \leq \lip(u,M)$. Assume by contradiction that there exists $C \in (\ell, \lip(u,M))$ and pick $z,w \in M$ such that $u(z) \geq u(w) + C\di(w,z)$. Then, 
\[
\ell = S^+_{u,\infty}(w) \geq \frac{u(z)-u(w)}{\di(w,z)} \geq C, 
\]
contradiction. The second equality follows from $S_{u,\infty}^+(x) = S_{u,\infty}^+(o)$ and the definition of $S^+_{u,\infty}(o)$.
\end{proof}

As we shall see, Lemma \ref{lip_const_S_R} guarantees the non-constancy of any blowdown of $u$. Thus, it plays the same important role as that of the relation 
\begin{equation}\label{eq_lim_Du_harm}
\lim_{R \to \infty}  \fint_{B_R} |\nabla u|^2 = \sup_M |\nabla u|^2
\end{equation}
in the theory of harmonic functions \cite{ccm,litam} and minimal graphs \cite{cgmr}. However, we emphasize that the proof of \eqref{eq_lim_Du_harm} in the above references is considerably subtler than that of Lemma \ref{lip_const_S_R}.

%
%\begin{lemma}\label{lem_lipbound}
%Let $M$ be complete and $u \in C(M)$ satisfy $\Delta_\infty u \ge 0$. If 
%\begin{eqnarray}
%\limsup_{r(q) \to \infty} \frac{u(q)}{r(q)} \leq c
%\end{eqnarray} 
%for some $c>0$, then $\lip(u,M) \le c$. 
%\end{lemma}
%
%\begin{proof}
%We fix $x \in M$, $c_1 > c_2 > c$ and let $c_3$ be big enough to satisfy $u(x) \le c_3 + c_2r(x)$ on $M$. We compare $u$ with the cone $C_x(y) = u(x) + c_1\di(x,y)$, so we compute
%\begin{eqnarray*}
%u(y) - C_x(y) &=& u(y) - u(x) - c_1\di(x,y) \leq c_3 + c_2\di(y,o) - u(x) - c_1\di(x,y) \\[0.2cm]
%&\leq & c_3 + c_2\di(x,o) + (c_2-c_1)\di(x,y) - u(x).
%\end{eqnarray*}
%Therefore, for some $R>0$ sufficiently large, if $y \not\in B_R(x)$ then $u(y) \leq C_x(y)$. We apply the comparison with cones in the ball $B_R(x)$ to deduce $u(y) \leq C_x(y)$ for every $y \in M$. Whence, $u(y) \leq u(x) + c_1\di(x,y)$ for each $x,y \in M$. Letting $c_1 \to c$ we conclude $\lip(u,M) \leq c$. 
%\end{proof}

\subsection*{Tightness and the anti-peeling Lemma}

We next present two key lemmata which will be often used in the arguments below. The first one adapts \cite[Lem. 4.2]{CEG}.

\begin{lemma}\label{lem_imposplitting}
Let $(N,\di_N)$ be a metric space and let $v$ be a $1$-Lipschitz function on the product space $\R \times N$ such that, for some $y_0 \in N$, 
\[
v(x,y_0) = x \qquad \forall \, x \in \R.
\]
Then, 
\[
v(x,y) = x \qquad \forall \, (x,y) \in \R \times N.
\]
\end{lemma}

\begin{proof}
Let us fix $\lambda \in \R$. Since $v$ is $1$-Lipschitz
\begin{eqnarray}
\vert v(x,y) - \lambda \vert^2 = \vert v(x,y) - v(\lambda,y_0)\vert^2 \leq \vert x - \lambda\vert^2 + \di_N(y,y_0)^2.
\end{eqnarray} 
Expanding the squares on both sides and simplifying we get
\begin{eqnarray}
v(x,y)^2 - 2\lambda v(x,y) \leq x^2 - 2\lambda x + \di_N(y,y_0)^2.
\end{eqnarray}
Dividing by $\lambda > 0$ and letting $\lambda \to + \infty$ we obtain $v(x,y) \geq x$. Likewise, dividing by $\lambda < 0$ and letting $\lambda \to - \infty$ we conclude that $v(x,y) \le x$, whence $v(x,y) = x$.
\end{proof}

The second Lemma follows from \cite[Prop. 6.2]{crandall_visit}. We borrowed the name ``anti-peeling Lemma" because of its analogy with \cite[Thm. 3.2]{bartnik_simon}, which is a key result in the theory of the prescribed Lorentzian mean curvature equation.

\begin{lemma}\label{lem_antipeeling}\emph{\textbf{[Anti-peeling Lemma]}}
Let $M$ be a complete Riemannian manifold, $\Omega \subset M$ an open subset, and let $u \colon \Omega \to \R$ satisfy, for some $x \in \Omega$,  
\[
\Delta_\infty u \ge 0 \ \ \text{ on } \, \Omega, \qquad S^+_u(x) = \|\nabla u \|_{L^\infty(\Omega)} = 1. 
\]
Then, there exists a segment $\gamma \colon [0,b) \to \Omega$ issuing from $x$ such that 
\begin{equation}\label{eq_u_slope}
u(\gamma(t)) - u(\gamma(s)) = t-s
\end{equation}
for each $0 < s < t < b$. Moreover, $u$ is differentiable at each point of $\gamma((0,b))$ with gradient $\nabla u(\gamma(t)) = \gamma'(t)$, and if $b< \infty$ it holds 
\[
\lim_{t \to b} \gamma(t) \in \partial \Omega.
\]
In particular, the existence of such $\gamma$ occurs if $\|\nabla u \|_{L^\infty(\Omega)} = 1$ and there exists a geodesic $\bar \gamma \colon [0,b') \to M$ issuing from $x$ where \eqref{eq_u_slope} holds for $0 < s < t < b'$, and in this case $\gamma$ extends $\bar \gamma$. 
\end{lemma}

\begin{proof}
By Proposition \ref{prop_lip_point}, $S^+_u(x)$ coincides with $\lim_{r\to 0} \lip(u,B_r(x))$. It was proved in \cite[Prop. 6.2]{crandall_visit} that there exists a Lipschitz curve $\gamma \colon [0,b) \to \Omega$ of velocity $|\gamma'|\le 1$ issuing from $x$ and satisfying, among other properties,  
\[
u(\gamma(t)) \ge u(x) + t S_u^+(x) = u(x) + t, \qquad \lim_{t \to b}\gamma(t) \in \partial \Omega \ \ \ \text{ if $b$ is finite}.
\]
Since $u$ is $1$-Lipschitz, $u(\gamma(t)) = u(x) + t$ on $[0,b)$ and therefore
\[
|t-s| = |u(\bar \gamma(t))- u(\bar \gamma(s))| \le \di(\gamma(t),\gamma(s)) \le |t-s|,
\]
whence $\gamma$ is a segment. By \cite[Lem. 3.5]{wang_notes}, if the domain of a $1$-Lipschitz function $u$ contains a segment $\gamma$ where $u$ has slope $1$, then $u$ is differentiable at any interior point of $\gamma$. Moreover, its gradient is $\pm \gamma'(t)$ according to whether $u$ grows or decreases along $\gamma$. This concludes the first part of the proof. Next, let $\bar \gamma \colon [0, b') \to \Omega$ be a geodesic from $x$ satisfying \eqref{eq_u_slope}. Using again \cite[Lem. 3.5]{wang_notes} and Proposition \ref{prop_lip_point} we get $S^+_u(y) = 1$ and $\nabla u(y) = \bar \gamma'(t)$ at every interior point $y = \bar \gamma(t)$ of $\bar \gamma$. Applying the first part of the proof, there exists a curve $\gamma$ issuing from $y$ where $u$ has slope $1$. Since $u$ is differentiable at $y$, $1 = (u\circ \gamma)'(0) = \langle \nabla u(y),\gamma'(0) \rangle \le 1$, whence $\bar \gamma'=\gamma'$ at $y$ and $\gamma$ extends $\bar \gamma$.
\end{proof}

\section{Manifolds with $\Ric \geq 0$}

We begin by investigating manifolds with $\Ric \ge 0$. First, we analyse their blowdowns by adapting an argument in \cite[Prop. 7.1]{crandall_visit}, see also Lemma 7.1 therein and \cite[Prop. 1]{hongzhao}.

\begin{theorem}\label{teo_ric}
Let $M^m$ be a complete manifold with $\Ric \geq 0$, and let $u \in C(M)$ be an $\infty$-harmonic function such that
\begin{eqnarray}
\limsup_{r(q) \to \infty} \frac{u(q)}{r(q)} =1, 
\end{eqnarray}
where $r$ is the distance from a fixed origin. Then, every tangent cone at infinity of $M$ splits as $\R\times N_\infty$ for some $N_\infty \in \RCD(0,m-1)$. Furthermore, the blowdown of $u$ only depends on the arclength $\tau$ of the $\R$-factor, and it is affine in $\tau$. 
\end{theorem}
\begin{proof}
By Lemma \ref{lip_const_S_R}, $\lip(u,M)=1$. Let $M_j^\lambda \to M_\infty^\lambda$ be a tangent cone at infinity centered at $o \in M$, and let $u_j^\lambda \to v^\lambda$ be the associated blowdown of $u$. We hereafter omit the superscript $\lambda$. Fix $R>0$, and for each $j$ consider a point $z_j^+ \in \partial B_{\lambda_j R}(o) \subset M$ which realizes $S_{\lambda_j R}^+(o)$. By Lemma \ref{lip_const_S_R}, 
\begin{eqnarray}
\frac{u_j(z^+_j)}{R} = \frac{u(z_j^+)-u(o)}{\lambda_j R} \to 1 \qquad \text{as} \ \ j \to \infty.
\end{eqnarray}
Likewise, we can consider $z_j^- \in \partial B_{\lambda_j R}(o) \subset M$ which realizes $S_{\lambda_j R}^-(o)$ and obtain
\[
\frac{u_j(z^-_j)}{R} = \frac{u(z_j^-)-u(o)}{\lambda_j R} \to -1 \qquad \text{as} \ \ j \to \infty.
\]
From $z^\pm_j \in \partial B_R^j(o)$ passing to limits as $j \to \infty$ and using the local uniform convergence of $u_j$, up to subsequences 
\begin{eqnarray}\label{u_segm}
z^\pm_j \to z_R^\pm \in \partial B_R^{\infty}(o_\infty), \qquad v(z_R^+) = R = -v(z_R^-).
\end{eqnarray}
Having set $\gamma^+_R \colon [0,R] \to M_\infty$ (respectively $\gamma^-_R \colon [0,R] \to M_\infty$) a segment from $o$ to $z_R^+$ (resp. from $o$ to $z_R^-$), we can define $\gamma_R \colon [-R,R] \to M_\infty$ as 
\begin{eqnarray}
\gamma_R(t) = 
\begin{cases}
\gamma_R^-(-t) & \text{for} \ \ t \in [-R,0], \\[0.2cm]
\gamma_R^+(t) & \text{for} \ \ t \in [0,R].
\end{cases}
\end{eqnarray}
From \eqref{u_segm} we deduce $\di_\infty (z_R^+,z_R^-) \geq u(z_R^+) - u(z_R^-) = 2R$, so by the triangle inequality $\di_\infty (z_R^+,z_R^-) = 2R$. It follows that $\gamma_R$ is a segment from $z_R^-$ to $z_R^+$, and by \eqref{u_segm} and $\lip(v,M_\infty) \le 1$ we deduce 
	\begin{equation}\label{eq_gammaRt}
	v(\gamma_R(t)) = t \qquad \text{for each } \, t \in [-R,R]. 
	\end{equation}
Letting $R \to \infty$, $\gamma_R$ converges to a line $\gamma_\infty$ in $M_\infty$. Cheeger-Colding's splitting Theorem in \cite[Thm. 6.64]{cheegercolding} guarantees that $M_\infty$ splits as $\R \times N_\infty$. Moreover, as shown by Gigli's nonsmooth splitting Theorem \cite{gigli}, $(N_\infty,\di') \in \RCD(0,m-1)$. Let $(\tau,y) \in \R \times N_\infty$, with $o = (0,o')$. Since $v(\tau,o') = \tau$ for each $\tau \in \R$, the conclusion $v(\tau,y)=\tau$ on $\R \times N_\infty$ then follows from Lemma \ref{lem_imposplitting}. 
\end{proof}

\begin{remark}\label{rem_useful}
\emph{Notice that the identity $v(x,y) = x$ for $(x,y) \in \R \times N_\infty$, together with \eqref{eq_gammaRt}, imply that each $\gamma_R$ is the curve $(t,o')$ for $t \in [-R,R]$. Hence, $\gamma_\infty$ is indeed the extension of each $\gamma_R$ to the entire real line. 
}
\end{remark}
 
As mentioned above, Theorem \ref{teo_ric} is not enough to guarantee that $M$ itself splits off a line. The following counterexample describes a manifold with $\Ric > 0$ (hence, not splitting off lines) and carrying a linearly growing $\infty$-harmonic function. Even more, the example points out that assumption $\Sec \ge 0$ cannot be weakened to the non-negativity of any of the following partial Ricci curvature functions $\Ric^{(\ell)}$ for $\ell \ge 2$:
\begin{definition}\label{def_Ricc_l}
Let $M$ be a manifold of dimension $m \ge 2$. For $\ell \in \{1,\ldots, m-1\}$, the $\ell$-th (normalized) Ricci curvature is the function
$$
v \in T_xM \quad \longmapsto \quad \Ric^{(\ell)}(v) \doteq  \inf_{\footnotesize{\begin{array}{c}
\mathcal{W} \le v^\perp \\
\dim \mathcal{W} = \ell
\end{array}}
} \left( \frac{1}{\ell} \sum_{j=1}^\ell \Sec(v \wedge e_j)\right), 
$$
where $\{e_j\}$ is an orthonormal basis of $\mathcal W$. 
\end{definition}
We recall that $\Ric^{(\ell)}$ interpolates between the sectional and Ricci curvatures, obtained respectively  for $\ell = 1$ and (up to a normalization constant) for $\ell = m-1$. In particular, with our chosen normalization the following implications are immediate:
$$
\Sec \ge \kappa \ \  \Longrightarrow \ \ \Ric^{(\ell-1)} \ge \kappa \ \ \Longrightarrow \ \ \Ric^{(\ell)} \ge \kappa \ \  \Longrightarrow \ \ \Ric \ge (m-1)\kappa.
$$

\begin{proposition}\label{teo_counter}
For $m \ge 4$, there exists a complete manifold $M$ with 
	$$
	\Ric^{(2)} \ge 0, \qquad \Ric > 0, \qquad \text{and} \qquad |\Sec| \le \bar \kappa^2
	$$
for some constant $\bar \kappa>0$, which carries a non-constant linearly growing $\infty$-harmonic function.
\end{proposition} 

\begin{proof}
We consider the example in \cite[p. 913]{kasuewashio}, along with the observations made in \cite{cgmr}. Fix $\alpha, \beta \in (0,1)$ such that $m-1-\beta > 2+\alpha$, and let $0 < \zeta_1,\zeta_2 \in C^\infty(\R^+)$ satisfy 
	\[
	\zeta_1(t) = \left\{ \begin{array}{ll}
	t & \text{ if } t \in (0,1] \\
	t^{-1-\alpha} & \text{ if } \, t \in [2,\infty), 
	\end{array}\right. \qquad  \zeta_2(t) = \int_t^\infty \zeta_1(s) \di s. 
	\]
Then, for $b,c \in \R^+$, define 
	\[
	\eta(r) = \frac{1}{2} r + \frac{1}{2\zeta_2(0)} \int_0^r \zeta_2(s) \di s, \qquad f(r) = (b+r^2)^{\frac{\beta + 3 - m}{2}} + c.
	\]		
We consider $M \doteq \R \times \R^+ \times \Sph^{m-2}$ with coordinates $(t,r,\theta)$ and metric
	\[
	g = f(r)^2 \di t^2 + \di r^2 + \eta(r)^2\di \theta^2,
	\]
where $\di \theta^2$ is the round metric on $\Sph^{m-2}$. Notice that the choice of $\eta$ implies that $g$ extends smoothly at $r=0$ and that $M$ is complete. It was shown in \cite[Section 9]{cgmr} that $|\Sec|$ is bounded on $M$, and that $\Ric^{(2)} \ge 0$, $\Ric > 0$ on $M$ if $b,c$ are chosen large enough. Given a function $u : M \to \R$ of the coordinate $t$ alone, it holds 
\[
|\nabla u| = \frac{\partial_t u}{f}, \qquad \Delta_\infty u = \frac{(\partial^2_t u)(\partial_t u)^2}{f^4}.
\]
Hence, any affine function $u(t) = at + k$ gives rise to an $\infty$-harmonic function. Also, $|\nabla u| = a/f$ is bounded on $M$ since $f$ is bounded below by a positive constant, thus $u$ has at most linear growth. 
\end{proof} 
  
Despite, in general, $\Ric \ge 0$ does not guarantee the splitting of $M$ when the latter supports a non-constant, linearly growing $\infty$-harmonic function, this happens in some special cases. The next result is a tangency principle between $\infty$-harmonic functions and (rescaled, translated) Busemann functions. 

\begin{proposition}\emph{\textbf{[tangency principle]}}\label{prop_SMP}
Let $M$ be a complete manifold with $\Ric \ge 0$, and let $u  \in C(M)$ satisfy $\Delta_\infty u \ge 0$ on $M$. Let $c>0$ and assume that there exists a ray $\gamma$ such that $u- cb_\gamma$ has a global maximum point. Then, 
\begin{itemize}
\item[(i)] $M$ splits as $\R \times N$, for some complete manifold $N$ with $\Ric_N \ge 0$; 
\item[(ii)] choosing a suitable arclength parameter $x$ of the $\R$-factor, it holds $u(x,y) = cx$ for each $(x,y) \in \R \times N$;
\item[(iii)] $\gamma$ is a half-line of the type $(-\infty,a] \times \{y_0\}$. In particular, $u-cb_\gamma$ is constant on $M$.
\end{itemize}
\end{proposition}

\begin{proof}
First, observe that for any fixed $o \in M$
\[
u(x) \le c b_\gamma(x) + \max_M (u-cb_\gamma) \le c(b_\gamma(x_0) + \di(x,o)) + \max_M (u-cb_\gamma),
\]
whence by Lemma \ref{lip_const_S_R} $u$ is $c$-Lipschitz. Let $p$ be a global maximum point of $u-cb_\gamma$. Defining 
\[
\bar u = \frac{1}{c}\big(u- u(p)+ cb_\gamma(p)\big)
\]
then $\bar u$ is $1$-Lipschitz and $\bar u \le b_\gamma$ on $M$, with equality in $p$. We consider an asymptote $\tilde \gamma$ for $\gamma$ at $p$. Then, $b_\gamma(p) + b_{\tilde \gamma}$ is a support function for $b_\gamma$ from above at $p$, which gives
\[
\bar u(\tilde \gamma(t)) \le b_\gamma(p) + b_{\tilde \gamma}(\tilde \gamma(t)) = b_\gamma(p) - t
\]
for $t \ge 0$, with equality at $t=0$. Since $\bar u$ is $1$-Lipschitz, necessarily 
\[
\bar u(\tilde \gamma(t)) = b_\gamma(p) - t
\]
whence $\bar u$ is linear with slope $1$ on $\tilde \gamma$. Since $\Delta \bar u \ge 0$, by the anti-peeling Lemma \ref{lem_antipeeling} $\tilde \gamma$ can be continued to a line $\tilde \gamma \colon \R \to M$ where $\bar u$ has slope $1$. The splitting theorem guarantees that $M$ splits as $\R \times N$ with the product metric $\di \tau^2 + g_N$, $p = (0,y_0) \in \R \times N$ and $\R \times \{y_0\}$ is the line $\tilde \gamma$. This shows $(i)$. Then, 
\[
\bar u(\tau,y_0) = b_\gamma(p) + \tau,
\]
thus by Lemma \ref{lem_imposplitting} we infer $\bar u(\tau,y)= b_\gamma(p) + \tau$ on $M$. Up to choosing $x = \tau + h$ for suitable constant $h$ and recalling the definition of $\bar u$, we get $u(x,y) = cx$ on $M$, which proves $(ii)$. To conclude, observe that since $\tilde \gamma$ is an asymptote of $\gamma$, then necessarily $\gamma$ is of the type $(-\infty,a] \times \{y_1\}$ for some $y_1$. Direct computation of $b_\gamma$ shows that $b_\gamma(\tau,y) = \tau-a$ and thus $\bar u-b_\gamma$, hence $u - cb_\gamma$, is constant.
\end{proof}

\begin{remark}
\emph{Note that the completeness assumption on $M$ is crucial. Indeed, on $\R^m \backslash \{0\}$ the function $-|x|$ is $\infty$-harmonic and $-|x|+x_1$ attains infinitely many maximum points (see also \cite[Exercise 2.9]{crandall_visit}).
}
\end{remark}

\section{Manifolds with $\Sec \geq 0$}

Let $u \in C(M)$ satisfy $\Delta_\infty u = 0$ and 
\begin{eqnarray}
\limsup_{r(q) \to \infty} \frac{u(q)}{r(q)} =1. 
\end{eqnarray}
By general theory, if $\Sec \ge 0$ then blowdowns at any fixed point $o$ are unique (see \cite[Lemma 3.4]{guijarro_kapo}), so we denote by $M_\infty$ the blowdown and write $M_j^\lambda \to M_\infty$. By Theorem \ref{teo_ric}, $\lambda$ induces a splitting $M_\infty=\R \times N_\infty$ along a line $\gamma_\infty^\lambda$ for which the blowdown $v^\lambda$ of $u$ writes as $v^\lambda(x,y) = x$, $o_\infty = (0, o'_\infty)$ and $\gamma^\lambda_\infty(t) = (t,o'_\infty)$. It is well-known that a splitting of $M_\infty$ induces, if $\Sec \ge 0$, a splitting of $M$ itself (see \cite[Thm. 4.6]{abfp} for a proof). For our purposes, it is convenient to include the proof in the following lemma, which regards the behaviour of $u$ along the split off line of $M$.  
\begin{lemma}\label{lem_blowgamma}
If $\Sec \ge 0$, the line $\gamma_\infty^\lambda(t) = (t,o_\infty') \in M_\infty = \R \times N_\infty$ induces a unique line $\gamma^\lambda: \R \to M$ passing through an origin $o \in M$ whose blowdown is $\gamma^\lambda_\infty$. Moreover, 
	\begin{equation}\label{eq_sequential}
	\frac{u(\gamma^\lambda(\lambda_jR))- u(\gamma^\lambda(-\lambda_jR))}{2\lambda_j R} \to 1 \qquad \text{as } \ \ j \to \infty.
	\end{equation}
\end{lemma}
\begin{proof} 
As we work for fixed $\lambda$, we omit its writing. Fix $R>0$. We refer to the proof of Theorem \ref{teo_ric} for the construction of $\gamma_\infty$ and for notation, so let $z_j^\pm \in M_j$ realize $S^\pm_{\lambda_j R}(o)$, let $z_R^\pm \in M_\infty$ be their limits on $M_\infty$ and let $\gamma_R : [-R,R]\to M_\infty$ be the segment built therein to join $z_R^-$ to $z_R^+$. By Remark \ref{rem_useful}, $z^\pm_R = \gamma_R(\pm R)= \gamma_\infty(\pm R)$. It follows that
	\begin{equation}\label{eq_zjpm}
	 M_j  \ni z_j^\pm  \longrightarrow \gamma_\infty(\pm R) \qquad \text{as } \, j \to \infty,
	\end{equation}
We select segments $\gamma_j^\pm : [0, \lambda_j R] \to M$ joining $o$ to $z_j^\pm$. Up to subsequence, $\gamma_j^\pm \to \gamma^\pm$ for some rays $\gamma^\pm : [0,\infty) \to M$. The concatenation 
	\[
	\gamma_j : = -\gamma_j^- * \gamma_j^+ = \left\{ \begin{array}{ll}
	\gamma_j^-(-t) & \quad \text{for } \, t \in [-\lambda_jR,0), \\[0.2cm] 
	\gamma_j^+(t) & \quad \text{for } \, t \in [0,\lambda_jR], 
	\end{array}\right.
	\]	
locally uniformly converges to $\gamma = -\gamma^- * \gamma^+ : \R \to M$. We prove that $\gamma$ is a line, so fix $S>0$ and $s \le S$. Then, by Toponogov's Theorem,
	\[
	\di(\gamma_j(-s), \gamma_j(s)) \ge \di(z_j^-,z_j^+) \frac{s}{\lambda_j R}.
	\]	
However, $\di(z_j^-,z_j^+) = \lambda_j \di_j(z_j^-,z_j^+) = 2\lambda_j R(1+o_j(1))$, whence
	\[
	\di(\gamma_j(-s), \gamma_j(s)) \ge 2s (1+o_j(1)) .
	\]
Therefore, the excess		
	\[
	0 \le \di(\gamma_j(-s), o) + \di (\gamma_j(s),o) - \di(\gamma_j(-s),\gamma_j(s)) \le 2s o_j(1) \le 2S o_j(1) 
	\]
converges to zero uniformly for $s \in [0,S]$, which proves that $\gamma$ is a line. We next point out that the blowdown of $\gamma$ is exactly $\gamma_\infty$. Applying the cosine law to the hinge $(o,\gamma,\gamma_j)$ and using that the angle 
	\[
	\sphericalangle(\dot\gamma_j^\pm(0),\dot \gamma^\pm(0)) \to 0 \qquad \text{as } \, j \to \infty,
	\]
we deduce
	\begin{equation}\label{eq_nice}
	\di(\gamma(\pm \lambda_jR),z_j^\pm)^2 \le 2(\lambda_jR)^2 - 2(\lambda_jR)^2 \cos \sphericalangle(\dot\gamma_j^\pm(0),\dot \gamma^\pm(0)) = o_j(\lambda_j^2R^2).
	\end{equation}
Rescaling, we get
	\[
	\di_j(\gamma(\pm \lambda_jR),z_j^\pm) \to 0 \qquad \text{as } \, j \to \infty, 
	\]
and by the triangle inequality and \eqref{eq_zjpm} we deduce 
	\[
	\gamma(\pm \lambda_jR) \subset M_j^\lambda \to \gamma_\infty( \pm R).
	\]	
Therefore, the blowdown of $\gamma$ (which is clearly a line in $M_\infty^\lambda$) restricted to $[-R,R]$ is a segment joining $\gamma_\infty(-R)$ to $\gamma_\infty(R)$. Since $\gamma_\infty$ is the only such segment, we conclude from the arbitrariness of $R$ that $\gamma$ blows down to $\gamma_\infty$. If there were a line $\sigma \neq \gamma$ with $\sigma(0)=o$ whose blowdown is $\gamma_\infty$, writing
	\[
	\sigma(s) = (b_1s,\bar \sigma(s)) \in \R \times N \qquad \text{with } \, |b_1| < 1,
	\]
the curve $\bar \sigma : \R \to N$ would be a line in $N$. Since $N$ has non-negative sectional curvature, the splitting theorem would guarantee that
	\[
	M = \R \times \R \times N',  \qquad \text{with } \quad \left\{ \begin{array}{ll}
	\gamma(t) = (t,0,o''), \\
	\sigma(s) = (b_1s, b_2 s, \hat \sigma(s))
	\end{array}\right. 
	\]	
for some $o'' \in N'$, $b_2 \in (0,1)$ and line $\hat \sigma$ in $N'$, which is incompatible with the assumption that $\sigma$ blows down to $\gamma_\infty$.

	To prove \eqref{eq_sequential}, observe that by definition of $z_j^\pm$, 
\begin{equation}\label{eq_u_jpm}
\frac{u_j(z^\pm_j)}{R} = \frac{u(z_j^\pm) - u(o)}{\lambda_jR} \to \pm 1 \qquad \text{as} \ \ j \to +\infty.
\end{equation}
We consider
\begin{equation}\label{eq_difference}
0 \le 1- \frac{u(\gamma(\lambda_jR))- u(\gamma(-\lambda_jR))}{2\lambda_j R} = 1 - \frac{u(z_j^+)- u(z_j^-)}{2\lambda_j R} + A_+ - A_-
\end{equation}
where 
\[
A_\pm = \frac{u(\gamma(\pm \lambda_jR))-u(z_j^\pm)}{2\lambda_j R}.
\]
Using $\lip(u,M) = 1$ and \eqref{eq_nice} we get
\[
|A_\pm| \le \frac{\di(\gamma(\pm \lambda_jR),z_j^\pm)}{2\lambda_j R} \to 0 \qquad \text{as } \, j \to \infty, 
\]
whence letting $j \to \infty$ in \eqref{eq_difference} and using \eqref{eq_u_jpm} we conclude \eqref{eq_sequential}. 
%Rescaling and passing to limits in \eqref{eq_sequential}, the blowdown $\bar \gamma_\infty$ of $\gamma$ therefore satisfies
%	\[
%	v(\bar \gamma_\infty(t)) = t \qquad \text{for } \, |t| \le R.
%	\]
%Hence, necessarily $\bar \gamma_\infty(t) = (t,o'_\infty)$ for $|t| \le R$. By the arbitrariness of $R$, we deduce that $\bar \gamma_\infty = \gamma_\infty$. 
\end{proof}

We first study the $2$-dimensional case, where Theorem \ref{teo_ric} and Savin's result \cite{savin} are enough to give a full classification. 

%\begin{theorem}\label{teo_sec_surface}
%Let $M$ be a complete surface with $\Sec \geq 0$, and let $u \in C(M)$ be a non-constant $\infty$-harmonic function such that
%\begin{eqnarray}
%\limsup_{r(q) \to \infty} \frac{u(q)}{r(q)} < \infty,
%\end{eqnarray}
%where $r$ is the distance from a fixed origin. Then, $M = \R^2$ or $M = \R \times \mathbb{S}^1$. Furthermore, $u$ only depends on the arclength $x$ of a split $\R$ factor, and it is affine in $x$. 
%\end{theorem}

\begin{proof}[Proof of Theorem \ref{teo_sec_surface_intro}]
By Theorem \ref{teo_ric} and Lemma \ref{lem_blowgamma}, $M = \R \times N$ for some $1$-dimensional complete manifold $N$, which is therefore either $\R$ or $\mathbb{S}^1$. If $M = \R^2$, since $u \in \lip(M)$ we can apply Savin's result to deduce that $u$ is affine. If $M = \R \times \mathbb{S}^1$, we consider the universal covering $\pi : \R \times \R \to \R \times \mathbb{S}^1$ and the preimage $\bar u = u \circ \pi$, which is $\infty$-harmonic. Savin's result guarantees that $\bar u$ is affine on $\R^2$, so $\bar u(x,y) = ax + by + c$. By construction, $\bar u$ is bounded in the $y$-coordinate for any fixed $x$, thus $b=0$ and $\bar u(x,y) = ax + c$ only depends on the first factor.
\end{proof}

In dimension $m \ge 3$, we are going to show uniqueness of the blowdown of $u$. This depends on refined one-sided gradient estimates for solutions to approximate problems which were obtained, in Euclidean setting, by Evans and Smart \cite{evsm}. We closely follow the approach therein but with a different approximation, as we use solutions $u_p$ to the $p$-Laplace equation  
\[
\left\{ \begin{array}{ll}
\Delta_p u_p = 0 & \quad \text{on } \, \Omega \Subset M, \\[0.2cm]
u_p = u & \quad \text{on } \, \partial \Omega  
\end{array}\right.
\]
in the limit $p \to \infty$. The main properties of $u_p$ are put off to the next section as not to interrupt the flow of the discourse. We begin with the following

\begin{lemma}\label{lemma_grad_estimate}
Let $ M = \R\times \R \times N'$ be a complete manifold and let $x^1, x^2, \pi : M \to \R$ be the natural projections onto the three factors. Consider a smooth function $u \colon M \to \R$ satisfying 
\[
\max_{B_T}\vert u - b_1 x^1 - b_2 x^2\vert < \eta T,
\]
for some $\eta, T >0$, where $B_T$ is a geodesic ball centered at $o = (0,0,\tilde o)$. Then, there exists an interior point $q_0 \in B_T$ such that
\begin{eqnarray}
\vert \partial_{x^i} u(q_0) - b_i\vert &\leq & 4\eta, \quad \text{for $i = 1,2$}, \\[0.2cm]
\vert \nabla^{N'} u(q_0)\vert &\leq & 4\eta.
\end{eqnarray} 
\end{lemma}
\begin{proof}
Consider the auxiliary function 
\[
w = b_1 x^1 + b_2 x^2 - 2\frac{\eta}{T}\rho^2, \qquad \rho(q) \doteq \dist_M(q,o).
\]
We observe that $(u-w)(o) < \eta T$ and $(u-w)(q) \geq \eta T$ for any $q \in \partial B_T$. Therefore there exists an interior minimum point $q_0 \in B_T$. If $\rho$ is smooth around $q_0$, the desired conclusion follows from $\nabla (u-w)(q_0) = 0$. Otherwise, we use Calabi's trick by considering a unit speed minimizing geodesic $\gamma$ from $o$ to $q_0$ and the function $\rho_\eps$ with $\rho_\eps(q) = \eps + \dist_M(\gamma(\eps),q)$. From $\rho_\eps \ge \rho$ on $M$ with equality at $q_0$, and since $\rho_\eps$ is smooth near $q_0$ as shown in \cite[end of Lemma 7.1.9]{petersen}, the conclusion follows as above by replacing $\rho$ with $\rho_\eps$. 
\end{proof}	
	
With the above preparation, we are ready to prove the main result of this section, which generalizes \cite[Proposition 2]{hongzhao}.

\begin{theorem}\label{teo_unique_blowdown}
Let $M$ be a complete manifold with $\Sec \geq 0$, and let $u \in C(M)$ be an $\infty$-harmonic function such that
\begin{eqnarray}
\limsup_{r(q) \to \infty} \frac{u(q)}{r(q)} =1,
\end{eqnarray}
where $r$ is the distance from a fixed origin. Then, for each $o \in M$ the blowdown $v_\infty$ of $u$ at $o$ does not depend on the chosen sequence, and there exists a splitting $M = (\R \times N, \di x^2 + g_N)$  for some complete manifold $(N,g_N)$ such that 
	\begin{equation}\label{eq_limit_x}
	\lim_{x \to +\infty} \frac{u(x,y)-u(-x,y)}{2x} \to 1 \qquad \forall \, y \in N.
	\end{equation}
Moreover, in the splitting $M_\infty = \R \times N_\infty$ induced by $M = \R \times N$, it holds $v_\infty(x,y) = x$. 	
\end{theorem}

\begin{proof}
We already know by Lemma \ref{lip_const_S_R} that $\lip(u,M)=1$. We introduce some notation. Let $M^\lambda_j \to M_\infty$ be a tangent cone with associated blowdown $u_j^\lambda \to v^\lambda$, and let $\R \times N$ be the splitting of $M$ induced by the line $\gamma^\lambda$. Write $o = (0,o') \in \R \times N$. Notice that, by Lemma \ref{lem_blowgamma}, $N_\infty$ is the tangent cone of $N$ at $o'$. Also, \eqref{eq_sequential} guarantees that \eqref{eq_limit_x} holds with $y = o'$ provided that the blowdown $v^\lambda$ does not depend on $\lambda$. Once \eqref{eq_limit_x} is shown for $y = o'$, its validity for any fixed $y$ immediately follows from the triangle inequality, since
\[
\big|u(x,y) - u(-x,y) - \big(u(x,o') - u(-x,o')\big) \big| \le 2 \di_N(y,o').
\]
To conclude the proof we assume, by contradiction, the existence of a sequence $\mu = \{\mu_k\} \to \infty$ such that the blowdown $u_k^\mu \to v^\mu$ associated to the tangent cone $M_k^\mu \to M_\infty$ satisfies $v^\mu \neq v^\lambda$. Consider the line $\gamma_\infty^\mu$ in $M_\infty$ induced by $\mu$. From $v^\mu(\gamma_\infty^\mu(s)) = s$ and $v^\mu \neq v^\lambda$, we deduce $\gamma^\mu_\infty \neq \gamma^\lambda_\infty$ and therefore, in the splitting $\R \times N_\infty$ induced by $\gamma^\lambda_\infty$, we can write
	\[
	\gamma^\mu_\infty(s) = (b_1 s, \sigma^\mu_{\infty}(s)), \qquad b_1^2 + |\dot \sigma^\mu_\infty|^2 = 1.
	\]	
Since $\sigma^\mu_{\infty}$ is a line as well, it induces a splitting $N_\infty = \R \times \tilde N_\infty$ and, by Lemma \ref{lem_blowgamma}, a corresponding splitting $N = \R \times \tilde N$. Summarizing, we can write
	\[
	M_\infty = \R \times \R \times \tilde N_\infty, \qquad o_\infty = (0,0,\tilde o_\infty)
	\]
with induced projections $(x^1_\infty, x^2_\infty, \pi_\infty) : M_\infty \to \R \times \R \times \tilde N_\infty$, and in these coordinates 
	\[
	\gamma^\lambda_\infty(t) = (t,0, \tilde o_\infty), \qquad \gamma^\mu_\infty(s) = (b_1s, b_2s, \tilde o_\infty),  \ \ \text{with } \, b_1^2 + b_2^2 = 1, \ \ b_2 \neq 0. 
	\]
Accordingly, 
	\[
	M = \R \times \R \times \tilde N, \qquad o = (0,0,\tilde o)
	\]
with projections $(x^1, x^2, \pi)$, and 
	\[
	\gamma^\lambda(t) = (t,0,\tilde o), \qquad \gamma^\mu(s) = (b_1s, b_2s, \tilde o),  \ \ \text{with } \, b_1^2 + b_2^2 = 1, \ \ b_2 \neq 0. 
	\]
In these coordinates
	\[
	v^\lambda = x^1_\infty, \qquad v^\mu = b_1 x^1_\infty + b_2 x^2_\infty,
	\]
whence 
	\[
	u_j^\lambda \to x^1_\infty, \qquad u_k^\mu \to b_1 x^1_\infty + b_2 x^2_\infty	
	\]
pointwise in the Gromov-Hausdorff sense. Also, associated to the tangent cone $M_j^\lambda \to M_\infty$ we define coordinates 
	\[
	x^1_{\lambda,j} = \frac{x^1}{\lambda_j}, \qquad x^2_{\lambda,j} = \frac{x^2}{\lambda_j},
	\] 
which are $1$-Lipschitz on $M^\lambda_j$. By construction, $x^i_{\lambda,j} \to x^i_\infty$ pointwise in the Gromov-Hausdorff sense for each $i \in \{1,2\}$. Likewise, the coordinates	
	\[
	x^1_{\mu,k} = \frac{x^1}{\mu_k} \qquad \text{and} \qquad x^2_{\mu,k} = \frac{x^2}{\mu_k} \qquad \text{on } \, M_k^\mu
	\] 
satisfy $x^i_{\mu,k} \to x^i_\infty$. 
%Define the cylinder
%	\[
%	C_{\lambda_j R} = \bar{B}_{\lambda_jR}(0,0) \times \bar{B}_{\lambda_jR}(\tilde o) \subset M = \R^2 \times \tilde N.
%	\] 
%%	C^{\lambda,j}_R = \bar{B}_R(0,0) \times \bar{B}_R(o') \subset M_j^\lambda = \Big(\R \times \R, (\di %x^1_{\lambda,j})^2 + \big(\di x^2_{\lambda,j})^2\Big) \times \tilde N_j^\lambda. 
%and notice that $C_{\lambda_j R}$ is the product of two balls of radius $R$ in the rescaled manifold $M^\lambda_j$, so that $C_{\lambda_j R} \to C_R^\infty$ in pmGH, where $C_R^\infty = \bar{B}_R((0,0)) \times \bar B_R(\tilde o_\infty) \in M_\infty = \R^2 \times \tilde N_\infty$. 
%
%By the definition of pointwise convergence with respect to Gromov-Hausdorff topology, $u^\lambda_j \to x^1_%\infty$ and $x_j^i \to x_\infty^i$ for $i \in \{1,2\}$ mean that for each $R$, each $q_\infty \in B_R^\infty$ and $q_j \in B^j_{R} \subset M^\lambda_j$ with $q_j \to q_\infty$, it holds 
%	\begin{equation}\label{eq_bonito}
%	u^\lambda_j(q_j) \to x^1_\infty(q_\infty), \qquad x_j^i(q_j) \to x_\infty^i(q_\infty) \quad \forall \, i \in \{1,2\}.
%	\end{equation}
It easily follows that, for each $\ell >0$, there exists $j_0(\ell)$ such that
	\begin{equation}\label{eq_firstblow}
	\max_{q \in B_{\lambda_j}} |u(q) - x^1(q)| < \ell \lambda_j \qquad \text{for } \, j \ge j_0(\ell).
	\end{equation}
Indeed, otherwise, there exist $\ell >0$ and points $q_j \in B_{\lambda_j} \subset M$ such that
	\[
	\ell \le |u^\lambda_j(q_j) - x^1_{\lambda,j}(q_j)|.
	\]
Up to a subsequence, $q_j \in B_1^j \subset M_j^\lambda$ converges to $q_\infty \in B_1^\infty$ and therefore
	\[
	\begin{array}{lcl}
	\ell & \le & \disp |u^\lambda_j(q_j) - x^1_{\lambda,j}(q_j)| \le |u^\lambda_j(q_j) - x^1_\infty(q_\infty)| + |x^1_\infty(q_\infty) - x^1_{\lambda,j}(q_j)| \\[0.2cm]
	& \to & 0 \qquad \text{as } \, j \to \infty,
	\end{array}
	\]
contradiction. Similarly to \eqref{eq_firstblow}, for the tangent cone $M^\mu_k \to M_\infty$ and for $\eta >0$ we obtain  
	\begin{equation}\label{eq_secondblow}
	\max_{q \in B_{\mu_k}} |u(q) - b_1x^1(q)- b_2x^2(q)| < \eta \mu_k \qquad \text{for } \, k \ge k_0(\eta).
	\end{equation}
Consider for each $p>2$ and $j$ the solution $u_{p,j}$ to 
	\[
	\left\{ \begin{array}{ll}
	\Delta_p u_{p,j} = 0 & \text{on } \, B_{2\lambda_j}, \\[0.4cm]
	u_{p,j} = u & \text{on } \, \partial B_{2\lambda_j}.
	\end{array}\right.
	\]
As recalled in Section \ref{sec_pLapla}, $u_{p,j} \to u$ uniformly on $\overline{B}_{2\lambda_j}$ as $p \to \infty$. Moreover, by \eqref{eq_firstblow} we get
	\begin{equation}\label{eq_ipoell}
	\ell \ge \lambda_j^{-1} \|u_{p,j} - x^1\|_{L^\infty(B_{\lambda_j})} \qquad \text{for $p \ge p_j$ large}.
	\end{equation}
Therefore, since $\Sec \ge 0$ and $u$ is $1$-Lipschitz, Theorems \ref{teo_gradesti_p} and \ref{teo_gradesti_oneside} with $p \ge \ell^{-1}$ guarantee the existence of a constant $C = C(m,u(o))$ such that	
	\[
	|\nabla u_{p,j}| \le C, \qquad |\nabla u_{p,j}|^2 \le \partial_1 u_{p,j} + C \ell^{\frac{1}{8}}(1+\ell)^\frac{5}{4} \quad \text{on } \, B_{\lambda_j/2}.
	\]
We now choose $\ell, \eta$. First, define
	\[
	\theta = 1-b_1 \in (0,1),
	\]
and let $\ell,\eta>0$ small enough to satisfy
	\[
	C\ell^{\frac{1}{8}}(1+\ell)^\frac{5}{4} < \frac{\theta}{4}, \qquad \eta < \frac{3\theta}{4 \cdot 28}.
	\]
Therefore,  
	\begin{equation}\label{eq_grad_ones}
	|\nabla u_{p,j}|^2 \le \partial_1 u_{p,j} + \frac{\theta}{4} \qquad \forall \, j \ge j_0(\ell).
	\end{equation}
For $k_0 = k_0(\eta)$ as in \eqref{eq_secondblow}, choose $j_1= j_1(\eta)$ such that $\lambda_{j_1} \ge 2\mu_{k_0}$ and let
	\[
	j_2(\ell,\eta) = \max\{ j_0(\ell),j_1(\eta)\}.
	\] 
We choose $j = j_2$ and write $u_p = u_{p,j_2}$ for ease of notation. From $B_{\mu_{k_0}} \subset B_{\lambda_{j_2}/2}$ and the uniform convergence $u_p \to u$ as $p \to \infty$, and from \eqref{eq_firstblow},  \eqref{eq_secondblow}, we infer the existence of $p_1 = p_1(\ell,\eta)$ for which  
\begin{equation}\label{eq_approx}
%\begin{array}{l}
%	\disp \max_{q \in B_{\lambda_{j_2}}} |u^\varepsilon(q) - x^1(q)| < \ell \lambda_{j_2} \qquad \text{for } %\, p \in (0,p_1) \\[0.4cm]
	\max_{q \in B_{\mu_{k_0}}} |u_p(q) - b_1x^1(q)- b_2x^2(q)| < \eta \mu_{k_0} \qquad \text{for } \, p \ge  p_1.
%\end{array}
\end{equation}
Using Lemma \ref{lem_blowgamma}, we get the existence of $q_0 \in B_{\mu_{k_0}}$ such that
	\[
	\begin{array}{lcl}
	|\partial_1 u_p(q_0)| \le b_1 + 4\eta, \\[0.3cm]
	|\nabla u_p(q_0)| \ge 1 - |(b_1-\partial_1 u_p) \partial_{x^1} + (b_2- \partial_2 u_p) \partial_{x^2} - \nabla^{N'} u_p(q_0)| \ge 1- 12\eta.  
	\end{array}
	\]
On the other hand, \eqref{eq_grad_ones} and $B_{\mu_{k_0}} \subset B_{\lambda_{j_2}/2}$ give $|\nabla u_p(q_0)|^2 \le \partial_1 u_p(q_0) + \theta/4$. Putting together the estimates we conclude
	\[
	1 - 24\eta \le (1-12\eta)^2 \le |\nabla u_p(q_0)|^2 \le b_1 + 4\eta + \frac{\theta}{4} = 1 - \frac{3\theta}{4} + 4\eta,
	\]
contradicting our choice for $\eta$. 		
\end{proof}

We conclude this section with a sufficient condition for the function $u$ in Theorem \ref{teo_unique_blowdown} to depend only on the variable $x$. The result below is inspired by the proof of  the ``Half-space theorem" in \cite[Thm. 4.1]{CEG}, which guarantees that a solution to $\Delta_\infty u \ge 0$ on $\R^m$ is affine provided that it lies below an affine function. However, our statement is significantly different: on the one hand, it only applies to $\infty$-harmonic functions, while on the other hand the extra condition we require is only localised on a single ray.

\begin{theorem}\label{teo_split_sec}
Let $M$ be a complete manifold with $\Sec \geq 0$, and let $u \in C(M)$ be an $\infty$-harmonic function such that
\begin{eqnarray}\label{eq_limit}
\limsup_{r(q) \to \infty} \frac{u(q)}{r(q)} =1,
\end{eqnarray}
where $r$ is the distance from a fixed origin. Assume that there exist a ray $\gamma$ and a constant $C$ such that either 
\[
u\big(\gamma(t)\big) \ge t - C \qquad \text{or} \qquad u\big(\gamma(t)\big) \le -t + C
\]
for each $t \in \R^+$. Then, there exists a splitting $M = (\R \times N, \di x^2 + g_N)$ for some complete manifold $(N,g_N)$ such that $u(x,y) = x$ for each $(x,y) \in \R \times N$.
\end{theorem}
\begin{proof}
In our assumptions, we know that $\lip(u,M) =1$. Define $o = \gamma(0)$. By Theorem \ref{teo_unique_blowdown}, $M$ splits as $\R \times N$ with metric $\di x^2 + g_N$ in such a way that \eqref{eq_limit_x} is satisfied. In particular, \eqref{eq_limit} holds both for $u$ and for $-u$. Up to changing the sign of $u$, we can therefore assume that 
\begin{equation}\label{eq_lb}
u\big(\gamma(t)\big) \le -t + C \qquad \forall \, t \in \R^+.
\end{equation}
We first claim that in coordinates $(x,y)$ we have $\gamma(t) = (-t,o')$. Indeed, let $M_\infty = \R \times N_\infty$ be the blow-down of $M$ at $o$, $o_\infty = (0,o'_\infty)$ its reference point and $u_\infty, \gamma_\infty$ the associated blowdowns of $u$ and $\gamma$. We know by Theorem \ref{teo_ric} that $u_\infty(x,y)=x$, with $x$ the arclength of the $\R$-factor properly oriented. Blowing down \eqref{eq_lb} and using that $u$, hence $u_\infty$, is $1$-Lipschitz we deduce $u_\infty(\gamma_\infty(t)) = -t$ for each $t \in \R^+$. Writing $\gamma_\infty(t) = (b_1t, \sigma_\infty(t)) \in \R \times N_\infty$ with $b_1^2 + |\sigma_\infty'|^2 =1$ we get $\gamma_\infty(t) = (-t,o'_\infty)$. Our claim follows by the uniqueness part in Lemma \ref{lem_blowgamma}. We therefore proved that
\begin{equation}\label{inss}
u(x,o') - x \le C \qquad \forall \, x \in (-\infty,0].
\end{equation}
Since $\lip(u,M)=1$, the function $x \mapsto \delta(x,y) \doteq u(x,y)-x$ is non-increasing on $\R$ (thus, \eqref{inss} holds for each $x \in \R$). Let us call $\delta(-\infty,y)$ its limit as $x \to -\infty$. By assumption, $\delta(-\infty,y_0)$ is finite. We prove that $\delta(-\infty,y)$ does not depend on $y$. First, since $u$ is $1$-Lipschitz we have
\[
|\delta(x,y) - \delta(x,y_0)| = |u(x,y)-u(x,y_0)| \le \di_N(y,y_0) \qquad \forall \, x \in \R,
\]
whence $\delta(-\infty,y)$ is finite for each $y$. Again since $u$ is $1$-Lipschitz,
\[
\begin{array}{lcl}
\disp |x + \delta(x,y) - \lambda - \delta(\lambda,y_0)|^2 & = & |u(x,y) - u(\lambda,y_0)|^2 \\[0.3cm]
& \le & (x-\lambda)^2 + \di_N(y,y_0)^2 .
\end{array}
\]
Expanding the squares and rearranging,
\[
\big[\delta(x,y)-\delta(\lambda,y_0)\big]^2 + 2(x-\lambda)(\delta(x,y)-\delta(\lambda,y_0)) \le \di_N(y,y_0)^2.  
\]
Discarding the first term on the left hand side, putting $x = 2\lambda<0$, dividing by $\lambda$ and letting $\lambda \to -\infty$ gives $\delta(-\infty,y)-\delta(-\infty, y_0) \ge 0$. On the other hand, letting $x = \lambda/2<0$, dividing by $\lambda$ and letting $\lambda \to -\infty$ gives $\delta(-\infty,y)-\delta(-\infty, y_0) \le 0$. Hence, $\delta(-\infty,y)=\delta(-\infty, y_0)$.

Up to translating $u$, we can therefore assume 
\[
\delta(-\infty, y) = 0 \qquad \text{for each } \, y \in N
\]
so that
$$
u(x,y) \leq x \quad \forall \, (x,y) \in M, \qquad \mbox{and} \qquad \lim\limits_{x\to -\infty}(x - u(x,y))=0.
$$
By contradiction, we assume that $u(x_0,y_0)< x_0$ for some $(x_0,y_0) \in M$.  Defining $v(x,y)\doteq u(x+x_0,y)-x_0$, we observe that $v$ is an $\infty$-harmonic function satisfying  
\begin{equation}\label{guanabara}
\lim\limits_{x\to -\infty}(x - v(x,y))=0 \quad \mbox{and} \quad v(x,y) \leq x, \quad \mbox{with} \quad v(0,y_0)<0.
\end{equation}
Fix $\mu>0$ such that $v(0,y_0) \leq -\mu$. Since $|\nabla v| \leq 1$ a.e. in $M$, we get
\begin{equation}\label{guanabara2}
v(x,y) \leq \sqrt{|x|^2+d_N(y,y_0)^2}-\mu.
\end{equation}
Next, for $R>r$ we consider the sphere and ball of radius $R$ in $\R \times N$ centered at $(-r,y_0)$:
$$
S_R=\left\{(x,y) \, \colon \; \sqrt{|x+r|^2+d_N(y,y_0)^2}=R \right\},
$$
and 
$$
\overline{B}_R=\left\{(x,y) \, \colon \; \sqrt{|x+r|^2+d_N(y,y_0)^2} \leq R \right\}. 
$$
In order to obtain an upper bound for $v(x,y)$ on $S_R$, on the one hand $v(x,y) \le x$, while on the other hand, by \eqref{guanabara2}, 
\begin{equation}\label{guanabara4}
\begin{array}{ccl}
v(x,y) & \le & -\mu + \sqrt{|x|^2+d_N(y,y_0)^2} \\[0.2cm]
& = & -\mu + \sqrt{(x+r)^2-2rx-r^2+d_N(y,y_0)^2} \\[0.2cm]
& = & - \mu + \sqrt{R^2-2rx-r^2} \qquad \text{on } \, S_R.
\end{array}
\end{equation}
Whence, for each $(x,y) \in S_R$ we have
\[
v(x,y) \le \max_{x \in [-R-r,R-r]} \min \left\{ x, - \mu + \sqrt{R^2-2rx-r^2} \right\}.
\]
Since the function $\sqrt{R^2-2rx-r^2}$ is decreasing in $x$, the maximum is attained when 
\[
x \in [-R-r,R-r] \qquad \text{solves} \qquad x = - \mu + \sqrt{R^2-2rx-r^2},
\]
that is, $x = -\mu - r + \sqrt{R^2 + 2\mu r}$. Concluding, 
$$
v(x,y) \leq \sqrt{R^2+2\mu r}-\mu-r = -r + \frac{\sqrt{R^2+2\mu r}-\mu}{R}\sqrt{|x+r|^2+d_N(y,y_0)^2} \qquad \text{on } \, S_R.
$$
The same inequality is also satisfied at the vertex $(-r,y_0)$. Therefore, by the comparison with cone property, 
$$
v(x,y) \leq -r + \frac{\sqrt{R^2+2\mu r}-\mu}{R}\sqrt{|x+r|^2+d_N(y,y_0)^2} \qquad \text{on } \, B_R.
$$
To conclude the proof, for $0 \leq s < r$ we choose $x=-s$ and $y=y_0$ to deduce
$$
u(-s,y_0) -(-s) \leq \left(-1+\frac{\sqrt{R^2+2\mu r}-\mu}{R}\right)(r-s).
$$
Taking $R=2r$ and letting $r \to \infty$, we infer
$$
u(-s,y_0) -(-s) \leq -\frac{\mu}{4}<0, 
$$
which implies by letting $s \to \infty$ that $\delta(-\infty,y_0) \leq - \mu/4$, a contradiction.
\end{proof}

\section{On the approximation via $p$-harmonic functions}\label{sec_pLapla}

Let $u \in C(M)$ solve $\Delta_\infty u = 0$ on $M$. In this section, we prove the relevant gradient and one-side gradient estimates for $p$-harmonic approximations of $u$ that are used in the proof of Theorem \ref{teo_unique_blowdown}. In Euclidean setting, the result is due to Evans and Smart \cite{evsm}, where it is used to infer the uniqueness of the blowup of $u$ at a given point and, consequently, the everywhere differentiability of $u$. Therein, for $\eps>0$ the authors choose to approximate $\Delta_\infty$ with the operator  
	\[
	\Delta_{\infty,\eps} \phi = \eps e^{- \frac{|\nabla \phi|^2}{2\eps}} \diver \left( e^\frac{|\nabla \phi|^2}{2\eps} \nabla \phi \right) = \Delta_\infty \phi + \eps \Delta \phi
	\]
and, for each $R>0$, the function $u$ with the solution $u_\eps$ to the problem
	\[
	\left\{ \begin{array}{ll}
	\Delta_{\infty,\eps} u_\eps =0 & \quad \text{on } \, B_R, \\[0.2cm]
	u_\eps = u & \quad \text{on } \, \partial B_R.
	\end{array}\right.
	\]
While their arguments can be adapted to the manifold setting, we prefer to approximate via the $p$-Laplace operator, in the hope that its peculiar features (notably its homogeneity) may be further exploited to get even sharper estimates in the direction of those required in \cite{evans_savin}.

Fix a smooth, relatively compact set $\Omega$, and for $p \in (1,\infty)$ consider the solution $u_p$ to 
	\begin{equation}\label{eq_ueps}
	\left\{ \begin{array}{ll}
	\Delta_p u_p \doteq \diver \left( |\nabla u_p|^{p-2} \nabla u_p \right) = 0 & \quad \text{on } \, \Omega, \\[0.2cm]
	u_p = u & \quad \text{on } \, \partial \Omega.
	\end{array}\right.
	\end{equation}
It is known by \cite{bhattacharya_dibenedetto_manfredi} and the uniqueness result in \cite{armstrong_smart,jensen} that $u_p \to u$ uniformly on $\overline{\Omega}$ as $p \to \infty$. Moreover, $u_p \in C^{1,\alpha}(\Omega)$ and it is $C^\infty$ on the open set $\{|\nabla u_p| > 0\}$. In a manifold setting, a sharp local gradient estimate for $p$-harmonic functions was obtained by Kotschwar and Ni in \cite[Thm. 1.1]{kotschwarni}, see also \cite{wangzhang}. While they were interested in the limit $p \to 1$, in our setting their estimate implies the following simpler one, which is enough for our purposes.

\begin{theorem}\label{teo_gradesti_p}
Assume that $u_p$ is $p$-harmonic on a ball $B_{2R}(o)\Subset M^m$, and that $\Sec \ge - \kappa^2$ on $B_{2R}(o)$. Then, 
\[
\sup_{B_R(o)} |\nabla \log u_p|^2 \le \frac{C_m}{(p-1)^2} \left( \frac{p^2}{R^2} + \frac{p \kappa}{R} + \kappa^2 \right),
\]
where $C_m$ only depends on $m$.
\end{theorem}

As a consequence, if $R \ge 1$ and $u_p$ solves \eqref{eq_ueps} on $\Omega = B_{2R}(o)$, then for $p \ge R$ we have 
\[
|\nabla u_p(x)| \le C_m \left(1 + \kappa\right)\frac{u(x)}{R} \qquad \forall \, x \in B_R(o).
\]
Theorem \ref{teo_gradesti_p} is obtained by a careful application of the improved version of Cheng-Yau's technique to a suitable Bochner formula for the $p$-Laplacian, which we now recall. Formally differentiating the $p$-Laplacian at $u_p$ we get the linearized operator
	\[
	\phi \longmapsto \left.\frac{\di}{\di t}\right|_{t=0} \Delta_p(u_p + t \phi) = \diver\left( A(\nabla u_p) \nabla \phi \right) 
	\] 
where, for a nonzero vector $X \in T_xM$, $A(X) : T_xM \to T_xM$ is the endomorphism 
	\[
	A(X) = |X|^{p-2} \left( (p-2)\langle \frac{X}{|X|}, \cdot \rangle \frac{X}{|X|} + \mathrm{id}\right). 			\]
Notice that $A(X)$ has eigenvalues $(p-1)|X|^{p-2}$ in direction $X$ and $|X|^{p-2}$ on $X^\perp$. For each $x \in \{ |\nabla u_p|>0\}$ we consider a local orthonormal frame $\{\nu, e_j\}$ with $\nu = \nabla u_p/|\nabla u_p|$ and $\{e_j\}$, $2 \le j \le m$ tangent to the level sets of $u_p$. Then the following Bochner formula in \cite[Prop. 2.14]{mrs} holds:
\begin{equation}\label{bochner}
\begin{array}{l}
\disp \frac{1}{2} \diver \left( A(\nabla u_p) \nabla |\nabla u_p|^2 \right) = \\[0.5cm]
\qquad = \disp |\nabla u_p|^{p-2} \Big\{ (p-1)u_{\nu\nu}^2 + p \sum_j u_{\nu j}^2 + \sum_{i,j} u_{ij}^2 + \Ric(\nabla u_p, \nabla u_p)\Big\},
\end{array}
\end{equation}
where $u_{\nu\nu},u_{\nu j},u_{ij}$ are the components of $\nabla^2 u_p$. For convenience, we also consider the normalized linearization
\[
\LL_p \phi = \left.\frac{\di}{\di t}\right|_{t=0} \frac{\Delta_p(u_p + t \phi)}{|\nabla (u_p + t \phi)|^{p-2}} = |\nabla u_p|^{2-p} \diver\left( A(\nabla u_p) \nabla \phi \right).
\]
So that the above Bochner formula simplifies to
\begin{equation}\label{eq_bochner_nablaup}
\begin{array}{lcl}
\frac{1}{2}\LL_p |\nabla u_p|^2 & = & \disp (p-1) u_{\nu\nu}^2 + p \sum_j u_{\nu j}^2 + \sum_{i,j} u_{ij}^2 + \Ric(\nabla u_p, \nabla u_p) \\[0.5cm]
& = & \disp (p-2)|\nabla^2 u_p(\nu)|^2 + |\nabla^2 u_p|^2 + \Ric(\nabla u_p, \nabla u_p) \\[0.3cm]
& \ge & \disp (p-1)|\nabla^2 u_p(\nu)|^2 + \Ric(\nabla u_p, \nabla u_p).
\end{array} 
\end{equation}
We shall rewrite $\LL_p$ in trace form. Let $\{e^a\}$, $1 \le a,b \le m$ be the dual coframe of $\{e_a\}$. The components of $A(\nabla u_p)$ satisfy
\[
A^a_b = |\nabla u|^{p-2}\left[ \delta^a_b + (p-2) \frac{u^au_b}{|\nabla u|^2}\right],
\] 
and expanding $\Delta_p u_p = 0$ we get
\[
\Delta u = - (p-2) u_{\nu\nu}. 
\]
Therefore, a computation gives
\[
[\diver A(\nabla u_p)]_b e^b = A^a_{b,a}e^b = 2(p-2)|\nabla u|^{p-3} u_{\nu j} e^j.
\]
It follows that in components
\begin{equation}\label{eq_Lphi}
\begin{array}{lcl}
\LL_p \phi & = & |\nabla u|^{2-p} \left( A^a_b \phi^b\right)_a = |\nabla u|^{2-p} A^a_b \phi_a^b + 2(p-2)|\nabla u|^{-1} u_{\nu j} \phi^j \\[0.3cm]
& = & \disp \left[ \delta^a_b + (p-2) \frac{u^au_b}{|\nabla u|^2}\right]\phi^b_a + 2(p-2)|\nabla u|^{-1} u_{\nu j}  \phi^j.
\end{array}
\end{equation}
Assume that $M$ splits as $\R \times N$ with coordinates $(x,y)$ and metric $\di x^2 + g_N$, and consider the function $\partial_x u_p = \langle \nabla u_p, \partial_x \rangle$. Since $\partial_x$ is a Killing field,  
\begin{equation}\label{eq_partial_up}
\LL_p (\partial_x u_p) = 0 \quad \text{on } \ \{ |\nabla u_p|>0\}.
\end{equation}
Furthermore, by using \eqref{eq_Lphi} we get
\begin{equation}\label{eq_Lphi_x}
|\LL_p x| \le 2|p-2||\nabla u|^{-1}|\nabla^2 u(\nu)|. 
\end{equation}

We next estimate $\LL_p \varphi^2$ when $\varphi$ is a cut-off depending on the distance from a fixed point. 

\begin{lemma}\label{lem_cutoff}
Assume that $\Sec \ge -\kappa^2$ in a ball $B_{2R} \Subset M$, for some $\kappa \in \R^+_0$. Then, there exists a function $\varphi \in \lip_c(B_{2R})$ such that $0 \le \varphi \le 1$, $\varphi \equiv 1$ on $B_R$ and, for $p \ge m$,
\[
|\nabla \varphi| \le \frac{C}{R}, \qquad \LL_p \varphi^2 \ge - Cp\varphi \big(\mathcal{C}_R + R^{-1}|\nabla u|^{-1}|\nabla^2 u(\nu)|\big)
\]
in the barrier sense, where $C$ is an absolute constant and $\mathcal{C}_R \doteq R^{-2}(1+\kappa R)$.
\end{lemma}

\begin{proof}
Let $\eta \in C^\infty_c([0,2))$ satisfy
\[
0 \le \eta \le 1, \quad \eta \equiv 1 \ \text{on } \, [0,1], \quad \eta' \le 0, \quad |\eta'| + |\eta''| \le C,
\] 
and let $\varphi(x) = \eta(r(x)/R)$ where $r$ is the distance from the center of the ball. Setting 
\[
{\rm tn}_{\kappa}(t) = \left\{ \begin{array}{ll}
\kappa \coth(\kappa t) & \quad \text{if } \, \kappa > 0, \\[0.2cm]
1/t & \quad \text{if } \, \kappa = 0, 
\end{array}\right.
\]
by the Hessian comparison theorem
\[
\nabla^2 r \le {\rm tn}_{\kappa}(r) \Big( \metric - \di r^2\Big)
\]
in the barrier (i.e. support) sense, see \cite[Lem. 12.2.4]{petersen}. Noting that $\eta' =0$ on $[0,1]$, we have
\begin{equation}\label{eq_hessiancomp}
\begin{array}{lcl}
\nabla^2 \varphi & = & \disp \eta''R^{-2} \di r^2 + \eta'R^{-1} \nabla^2 r \ge - CR^{-2} \di r^2 + \eta' R^{-1} {\rm tn}_\kappa(R)\big( \metric - \di r^2\big) \\[0.2cm]
 & \ge & \disp -CR^{-2}(1 + R{\rm tn}_{\kappa}(R)) \metric \ge - C\mathcal{C}_R \metric
\end{array}
\end{equation}
for suitable $C$, where we used that $R{\rm tn}_{\kappa}(R) \le 1+ \kappa R$. Whence, using \eqref{eq_Lphi} and for $p \ge m$,
\begin{equation}\label{Leps_phi}
\begin{array}{lcl}
\LL_p \varphi & = & \disp \left[ \delta^a_b + (p-2) \frac{u^au_b}{|\nabla u|^2}\right]\varphi^b_a + 2(p-2)|\nabla u|^{-1} u_{\nu j}  \varphi^j \\[0.5cm]
& = & \disp (p-1)\varphi_{\nu\nu} + \varphi_{jj} + 2(p-2)|\nabla u|^{-1} u_{\nu j}  \varphi^j \\[0.3cm]
& \ge & \disp -C \mathcal{C}_R p - 2(p-2)|\nabla u|^{-1}|\nabla^2 u(\nu)||\nabla \varphi| \\[0.3cm]
& \ge & \disp -C p\big(\mathcal{C}_R + R^{-1}|\nabla u|^{-1}|\nabla^2 u(\nu)|\big) 
\end{array}
\end{equation}
and
\begin{equation}
\LL_p \varphi^2  \ge  2 \varphi \LL_p \varphi,
\end{equation}
from which the thesis follows. 
\end{proof}

We now prove the one-side gradient estimate.

\begin{theorem}\label{teo_gradesti_oneside}
Let $M^m = \R \times N$ be a complete manifold with metric $\di x^2 + g_N$, and assume that $N$ has $\Sec \ge - \kappa^2$. Let $u_p \in C^1(\overline{B}_{2R})$ solve $\Delta_p u_p = 0$ on $B_{2R}$. Then, for
	\[
	\ell \ge R^{-1} \| u-x\|_{L^\infty(B_{2R})}, \qquad A \ge 1 + \|\nabla u_p\|_{L^\infty(B_{2R})},
	\]
and $p \ge 2$ there exists an absolute constant $C$ such that
	\[
	|\nabla u_p|^2 \le \partial_x u_p + Cm A^2 \ell^{\frac{1}{8}}(\ell+1)^{\frac{5}{4}}\left[1+ \kappa R + \frac{1}{\ell p}\right]
 \qquad \text{on } \, B_R.
	\]
\end{theorem}

\begin{proof}
Let $\varphi$ be a cut-off function. For convenience, we suppress the subscript $p$ and simply write $u$. We consider 
\[
\Phi = |\nabla u|^2 - \partial_x u.
\]

Because of \eqref{eq_bochner_nablaup}, \eqref{eq_partial_up} and our assumptions on the sectional curvature, on the set $\{\Phi >0\}$ and setting $\nu = \nabla u/|\nabla u|$ it holds 
\[
\LL_p \Phi^2 \ge 2\Phi \LL_p \Phi \ge 4(p-1)\Phi \vert \nabla^2 u(\nu)\vert^2 - 4(m-1)\kappa^2 \Phi |\nabla u|^2.
\]
We compute on $\{\Phi>0\}$ the following expression:
\[
\begin{array}{lcl}
\LL_p (\varphi^2 \Phi^2) & \ge & \varphi^2 \LL_p \Phi^2 + \Phi^2 \LL_p \varphi^2 + 2(p-2) \langle \nu, \nabla \varphi^2 \rangle \langle \nu, \nabla \Phi^2 \rangle + 2\langle \nabla \varphi^2, \nabla \Phi^2 \rangle \\[0.3cm]
& \ge & \disp 4(p-1)\varphi^2 \Phi \vert \nabla^2 u(\nu)\vert^2 -4(m-1)\kappa^2\varphi^2 \Phi |\nabla u|^2  + \Phi^2 \LL_p \varphi^2 \\[0.3cm]
& & \disp - 8(p-1)\Phi \varphi |\nabla \varphi| |\langle \nu, \nabla \Phi \rangle| - 8 \Phi \varphi |\nabla \varphi| |\nabla \Phi|. 
\end{array}
\]
Notice that $\nabla \Phi = 2 \nabla^2 u(\nabla u) - \nabla^2 u(e_1)$, whence
\[
|\langle \nu, \nabla \Phi \rangle | \le 2|\nabla^2 u(\nu)| (1 + |\nabla u|), \qquad |\nabla \Phi| \le 2|\nabla^2 u| (1+|\nabla u|).
\]
Inserting into the above we get
%\[
%\Ric(\nabla u, \nabla u) = \Ric( \nabla(u-x_1), \nabla (u-x_1)) \ge - (m-1)\kappa^2|\nabla(u-x_1)|^2
%\]
\[
\begin{array}{lcl}
\LL_p (\varphi^2 \Phi^2) 
& \ge & \disp 4(p-1)\varphi^2 \Phi \vert \nabla^2 u(\nu)\vert^2 -4(m-1)\kappa^2\varphi^2 \Phi |\nabla u|^2  + \Phi^2 \LL_p \varphi^2 \\[0.3cm]
& & \disp - 16(p-1)\Phi \varphi |\nabla \varphi||\nabla^2 u(\nu)|(1 + |\nabla u|) - 16 \Phi \varphi |\nabla \varphi| |\nabla^2 u|(1+|\nabla u|) .
\end{array}
\]
On the other hand, we compute
\begin{align*}
\LL_p(u-x)^2& \ge 2(u-x)\LL_p(u-x)+2(p-1) \langle \nu,\nabla(u-x)\rangle^2\\[0.2cm]
&\ge -2 |u-x||\LL_p x| +2(p-1)|\nabla u|^{-2}\left(|\nabla u|^2-\partial_x u\right)^2 \\[0.2cm]
&\ge -4|p-2||u-x||\nabla u|^{-1}|\nabla^2 u(\nu)| +2(p-1)|\nabla u|^{-2}\Phi^2.
\end{align*}
Let us define 
\[
w \doteq \varphi^2\Phi^2+ \beta R^{-2}(u-x)^2+\ell|\nabla u|^2, 
\]
for some $\beta>0$ to be determined later. Notice that $\beta, \ell$ are invariant under the natural scaling $\metric' = R^{-2} \metric$ and $u' = u/R$. Let us assume that $w$ attains its maximum at an interior point $p_0$. If $\Phi(p_0)>0$, then
\begin{align*}
0\ge&\ \LL_p w =\LL_p(\varphi^2\Phi^2)+\beta R^{-2} \LL_p(u-x)^2+\ell \LL_p|\nabla u|^2\\[0.2cm]
\ge&\ 4(p-1)\varphi^2 \Phi \vert \nabla^2 u(\nu)\vert^2 - 4(m-1)\kappa^2 \varphi^2 \Phi |\nabla u|^2 + \Phi^2 \LL_p \varphi^2\\[0.2cm]
& - 16(p-1)\Phi \varphi |\nabla \varphi| |\nabla^2 u(\nu)|(1+|\nabla u|) - 16 \Phi \varphi |\nabla \varphi| |\nabla^2 u| (1+|\nabla u|)\\[0.2cm]
&-4|p-2|\beta R^{-2} |u-x||\nabla u|^{-1}|\nabla^2u(\nu)| +2(p-1)\beta R^{-2}|\nabla u|^{-2}\Phi^2\\[0.2cm]
&+2(p-2)\ell|\nabla^2 u(\nu)|^2 + 2 \ell |\nabla^2 u|^2 - 2(m-1)\kappa^2\ell |\nabla u|^2.
\end{align*}
Using that
\[
\begin{array}{l}
16 \Phi \varphi |\nabla \varphi| |\nabla^2 u(\nu)|(1+|\nabla u|) \le 2\varphi^2 \Phi |\nabla^2 u(\nu)|^2 + 32 \Phi |\nabla \varphi|^2 (1+ |\nabla u|)^2, \\[0.3cm]
16 \Phi \varphi |\nabla \varphi| |\nabla^2 u| (1+|\nabla u|) \le 2 \ell |\nabla^2 u|^2 + 32 \ell^{-1} \Phi^2 \varphi^2 |\nabla \varphi|^2 (1+|\nabla u|)^2, \\[0.3cm]
4\beta R^{-2} |u-x||\nabla u|^{-1}|\nabla^2u(\nu)| \le 2 \ell |\nabla^2 u(\nu)|^2 + 2 \beta^2 R^{-4}\ell^{-1} (u-x)^2 |\nabla u|^{-2},
\end{array}
\]
and the definition of $\Phi$ we get 
\begin{align*}
0\ge&\ 2(p-1)\varphi^2 \Phi |\nabla^2 u(\nu)|^2 - 4 (m-1)\kappa^2 \varphi^2 \Phi |\nabla u|^2 + \Phi^2 \LL_p \varphi^2\\[0.2cm]
&- 32(p-1) \Phi |\nabla \varphi|^2 (1+ |\nabla u|)^2 - 32 \ell^{-1} \Phi^2 \varphi^2 |\nabla \varphi|^2 (1+|\nabla u|)^2\\[0.2cm]
&- 2 \beta^2 R^{-4}\ell^{-1} (u-x)^2|\nabla u|^{-2} +2(p-1)\beta R^{-2}|\nabla u|^{-2} \Phi^2\\[0.2cm]
& - 2(m-1)\kappa^2\ell |\nabla u|^2.
\end{align*}
We hereafter denote with $C_1,C_2,\ldots$ absolute constants. Define $\varphi$ as in Lemma \ref{lem_cutoff}, so that $|\nabla \varphi|^2 \le C R^{-2}$ and, since $p \ge 2$, 
	\[
	\begin{array}{lcl}
	\LL_p \varphi^2 & \ge & \disp - Cp \varphi \mathcal{C}_R - Cp\varphi R^{-1} |\nabla u|^{-1} \vert \nabla^2 u(\nu)\vert \\[0.3cm]
	& \ge & - Cp \varphi \mathcal{C}_R - 2(p-1)\varphi^2 \Phi^{-1} |\nabla^2 u(\nu)|^2 - C_1 p \Phi R^{-2}|\nabla u|^{-2}.
	\end{array}
	\] 
Inserting into the above and multiplying by $|\nabla u|^2$ we infer
\begin{align*}
0\ge& - 4 (m-1)\kappa^2 \varphi^2 \Phi |\nabla u|^4 - Cp \varphi \Phi^2 \mathcal{C}_R|\nabla u|^2 - C_1 p \Phi^3 R^{-2} \\[0.2cm]
&- 32(p-1) \Phi |\nabla \varphi|^2 (1+ |\nabla u|)^2||\nabla u|^2 - 32 \ell^{-1} \Phi^2 \varphi^2 |\nabla \varphi|^2 (1+|\nabla u|)^2|\nabla u|^2\\[0.2cm]
&- 2 \beta^2 R^{-4}\ell^{-1} (u-x)^2 +2(p-1)\beta R^{-2} \Phi^2 - 2(m-1)\kappa^2\ell |\nabla u|^4.
\end{align*}
Using $|\varphi| \le 1$, $|\nabla \varphi| \le C_2/R$, $|u-x| \le \ell R$, $|\nabla u| \le A$ together with the inequality $\Phi^3 \le A^2 \Phi^2$, and rearranging, we obtain
\begin{align*}
0\ge& - 4 (m-1)\kappa^2 \Phi A^4 - Cp \Phi^2 \mathcal{C}_RA^2 - C_1 p A^2\Phi^2 R^{-2} \\
&- 32(p-1) \Phi C_2^2 R^{-2}A^4 - 32 \ell^{-1} \Phi^2 C_2^2R^{-2}A^4\\
&- 2 \beta^2 R^{-2}\ell + 2(p-1)\beta R^{-2} \Phi^2 - 2(m-1)\kappa^2\ell A^4 \\
= & \ \ (E\Phi^2 - B \Phi - F)R^{-2}, 
\end{align*}
where we set 
\[
\begin{array}{lcl}
E & \doteq & \disp 2(p-1)\beta -Cp R^2\mathcal{C}_R A^2 - C_1 p A^2 - 32 \ell^{-1}C_2^2 A^4\\[0.2cm]
B & \doteq & 4(m-1)\kappa^2R^2A^4 + 32(p-1)C_2^2 A^4, \\[0.2cm]
F & \doteq & \disp 2 \beta^2 \ell + 2(m-1)\kappa^2 \ell A^4 R^2.
\end{array}
\]
Whence, 
\[
E\Phi^2 \le B \Phi + F \le BA^2 + F \qquad \text{at } \, p_0.
\]
It follows that, at $p_0$,  
\[
\begin{array}{lcl}
w & \le & \disp \Phi^2 + \beta R^{-2} (u-x)^2 + \ell |\nabla u|^2 \\[0.2cm]
& \le & \disp E^{-1}(BA^2 + F) + \beta \ell^2 + \ell A^2.
\end{array}
\]
Therefore, for each $x \in B_R$, 
\begin{equation}\label{eq_estim_Phi}
\Phi^2(x) \le w(x) \le w(p_0) \le E^{-1}(BA^2 + F) + \ell(\beta \ell + A^2). 
\end{equation}
By the definition of $\mathcal{C}_R$
\[
E \ge p\beta - C_3p A^2\left[1+ \kappa R + \frac{A^2}{\ell p} \right] \ge \frac{p\beta}{2}, 
\]
where in the latter inequality we have chosen
\[
\beta = 2C_3 \left(\frac{\ell+1}{\ell}\right)^{\frac{1}{4}} A^4\left[1+ \kappa R + \ell^{-1}p^{-1}\right]. 
\]
Estimating $B,F$ for such a choice of $\beta$ we obtain:
\[
B \le C_5 mpA^4(1 + \kappa^2 R^2), \qquad F \le C_6 m \sqrt{\ell(\ell+1)}(1 + \kappa^2 R^2 + \ell^{-2}p^{-2})A^8,
\]
which gives
\[
\begin{array}{lcl}
\Phi^2(x) & \le & C_7 m A^4 \ell^{\frac{1}{4}} (\ell+1)^{\frac{3}{4}}(1 + \kappa R + \ell^{-1}p^{-1}) \\[0.2cm]
& \le & C_7 m A^4 \ell^{\frac{1}{4}} (\ell+1)^{\frac{5}{2}}(1 + \kappa R + \ell^{-1}p^{-1}) 
\end{array}
\]
%\[
%\begin{array}{lcl}
%\Phi^2(x) & \le & \disp C_3 \left(\frac{\ell+1}{\ell}\right)^{\frac{1}{4}} A^4\left[1+ \kappa R + \ell^{-1}p^{-1}\right] \cdot mpA^6(1 + \kappa^2 R^2) + C_6 m \sqrt{\ell(\ell+1)}(1 + \kappa^2 R^2 + \ell^{-2}p^{-2})A^8 \\[0.3cm]
%\Phi^2(x) & \le & \disp C_3  \cdot m A^4 \left(\frac{\ell}{\ell+1}\right)^{\frac{1}{4}}(1 + \kappa R) + C_6 m p^{-1} \left(\frac{\ell}{\ell+1}\right)^{\frac{1}{4}}\sqrt{\ell(\ell+1)}(1 + \kappa R + \ell^{-1}p^{-1})A^4 \\[0.3cm]
%\Phi^2(x) & \le & \disp C m A^4 \ell^{\frac{1}{4}} (\ell+1)^{\frac{3}{4}}(1 + \kappa R + \ell^{-1}p^{-1})
%\end{array} 
%\]
%
%
%Choose therefore
%\[
%\beta = \left[\frac{\ell+1}{\ell}\right]^{\frac{1}{2}} CA^2\left[1+ \kappa R + \frac{\eps}{\ell} \right], %\qquad \text{so that } \, E \ge \beta A^{-2}.
%\] 
%Moreover, estimating $B,F$ as follows:
%\[
%B \le C_m(1 + \eps \kappa^2 R^2), \qquad F \le C_m \ell(\beta^2 + \eps \kappa^2 R^2)
%\]
%we conclude, also using $A \ge 1$, 
%\[
%\begin{array}{lcl}
%\Phi^2(x) & \le & \disp C_m A^2 \beta^{-1}(1 + \eps \kappa^2 R^2(1+\ell) + \ell\beta^2)A^2 + \ell (\beta \ell + A^2) \\[0.3cm]
%& \le & C_m \left[\frac{\ell}{\ell+1}\right]^{\frac{1}{2}} A^6 \left(1 + \kappa R + \frac{\eps}{\ell}\right)(1+\ell) + C \ell^{\frac{3}{2}} \sqrt{\ell+1} A^2 \left(1 + \kappa R + \frac{\eps}{\ell}\right)\\[0.4cm]
%& \le & \disp C_m \sqrt{\ell}(1+\ell)^{\frac{3}{2}} A^6\left(1 + \kappa R + \frac{\eps}{\ell}\right) \\[0.4cm]
%\end{array} 
%\]
and the conclusion follows by taking square roots. If $p_0 \in \partial B_{2R}$ we have for each $x \in B_R$
\[
\begin{array}{lcl}
\Phi^2(x) & \le & \disp w(x) \le w(p_0) = \left(\beta R^{-2}(u-x)^2+\ell|\nabla u|^2\right)(p_0) \le \beta \ell^2 + \ell A^2  \\[0.2cm]
& \le & C_8 A^4 \ell^{\frac{7}{4}}(\ell+1)^{\frac{3}{4}}\left[1+ \kappa R + \ell^{-1}p^{-1}\right] \\[0.2cm]
& \le & C_8 A^4 \ell^{\frac{1}{4}}(\ell+1)^{\frac{5}{2}}\left[1+ \kappa R + \ell^{-1}p^{-1}\right],
\end{array}
\]
%\[
%\begin{array}{lcl}
%\Phi^2(x) & \le & \disp w(x) \le w(p_0) = \left(\beta R^{-2}(u-x)^2+\ell|\nabla u|^2\right)(p_0) \le \beta \ell^2 + \ell A^2  \\[0.2cm]
%& \le & C_m A^2 \ell^{\frac{3}{2}} \sqrt{\ell+1}\left(1 + \kappa R + \frac{\eps}{\ell}\right),
%\end{array}
%\]
%
from which the desired inequality follows as well. 
\end{proof}

\bibliographystyle{amsplain}

\end{document}